\documentclass[11pt]{amsart}
\usepackage{amssymb, amsmath, amsthm,commath,stackrel,subfiles,a4wide,enumerate}
\usepackage{color}
\usepackage[utf8]{inputenc}
\usepackage{graphicx,hyperref}
\usepackage{tikz-cd}
\usepackage{mathtools}
\usepackage{dsfont,bbm}
\usetikzlibrary{decorations.pathmorphing}

\tikzcdset{arrow style=tikz}

\theoremstyle{definition}
\newtheorem{example}{Example}[section]
\theoremstyle{plain}
\newtheorem{theorem}[example]{Theorem}
\newtheorem{lemma}[example]{Lemma}
\newtheorem{proposition}[example]{Proposition}

\newtheorem{corollary}[example]{Corollary}
\newtheorem{notation}[example]{Notation}

\theoremstyle{remark}
\newtheorem{remark}[example]{Remark}
\theoremstyle{definition}
\newtheorem{definition}[example]{Definition}
\allowdisplaybreaks	
\renewcommand{\phi}{\varphi}
\renewcommand{\epsilon}{\varepsilon}
\renewcommand{\flat}{\text{flat}}
\DeclareMathOperator{\ev}{ev}
\DeclareMathOperator{\GL}{GL}
\DeclareMathOperator{\Rep}{Rep}
\DeclareMathOperator{\Tr}{Tr}

\renewcommand{\Vec}{\Vecja}
\DeclareMathOperator{\Vecja}{Vec}
\DeclareMathOperator{\End}{End}
\DeclareMathOperator{\Mod}{Mod}
\DeclareMathOperator{\id}{id}
\DeclareMathOperator{\Hom}{Hom}
\DeclareMathOperator{\coker}{coker}
\DeclareMathOperator{\im}{im}
\DeclareMathOperator{\coim}{coim}
\DeclareMathOperator{\tensor}{\otimes}

\DeclareMathOperator{\re}{Re}
\DeclareMathOperator{\Aut}{Aut}
\DeclareMathOperator{\HS}{HS}
\DeclareMathOperator{\Omgb}{\Omega
}

\def\tilde{\widetilde}
\def\A{\mathcal A}
\def\B{\mathcal B}
\def\C{\mathcal C}
\def\dg{\mathrm{dg}}
\def\E{\mathcal E}
\def\F{\mathcal F}
\def\G{\mathcal G}

\def\S{\mathcal S}

\makeatletter
\newcommand{\exterior}{\@ifnextchar^\@extp{\@extp^{\,}}}
\def\@extp^#1{\mathop{\bigwedge\nolimits^{\!#1}}}
\makeatother

\begin{document}

  \title{The fundamental group of a noncommutative space}
  \author{Walter D. van Suijlekom}
\date{\today}
  \address{Institute for Mathematics, Astrophysics and Particle Physics, Radboud University Nijmegen, Heyendaalseweg 135, 6525 AJ Nijmegen, The Netherlands}
  \email{waltervs@math.ru.nl}
  \author{Jeroen Winkel}
\address{Mathematisches Institut der Westf\"alischen Wilhelms-Universit\"at M\"unster, Einsteinstrasse 62, D-48149 M\"unster, Germany}
\email{jwinkel@uni-muenster.de}

\begin{abstract}
  We introduce and analyse a general notion of fundamental group for noncommutative spaces, described by differential graded algebras. For this we consider connections on finitely generated projective bimodules over differential graded algebras and show that the category of flat connections on such modules forms a Tannakian category. As such this category can be realised as the category of representations of an affine group scheme $G$, which in the classical case is (the pro-algebraic completion of) the usual fundamental group. This motivates us to define $G$ to be the fundamental group of the noncommutative space under consideration. The needed assumptions on the differential graded algebra are rather mild and completely natural in the context of noncommutative differential geometry. We establish the appropriate functorial properties, homotopy and Morita invariance of this fundamental group. As an example we find that the fundamental group of the noncommutative torus can be described as the algebraic hull of the topological group $(\mathbb Z+\theta\mathbb Z)^2$. 
\end{abstract}

  \maketitle

\section{Introduction}
The fundamental group is one of the first tools used in algebraic topology to collect information about the shape of a topological space or of a manifold. Given the simplicity of its definition, it is quite surprising that no analogue of this group has been found yet for {\em noncommutative} spaces, especially in the context of noncommutative differential geometry \cite{connes}. This is in contrast to other structures on topological spaces that have found their counterparts in terms of noncommutative ($C^*$)-algebras, such as de Rham cohomology (in terms of cyclic cohomology), topological K-theory (as K-theory for $C^*$-algebras), Riemannian metrics (as spectral triples \cite[Ch. VI]{connes}) and measures (as von Neumann algebras \cite[Ch.V]{connes}).

This paper aims for a definition of the fundamental group of a noncommutative space. Of course, in view of Gelfand duality one could try to dualise based loops in terms of $*$-homomorphisms from a $C^*$-algebra $A$ to $C([0,1])$. However, such a homomorphism $\phi$ would send any commutator $ab-ba$ to 0, which for many noncommutative spaces already means that $\phi$ is trivial. Instead, in the spirit of \cite{Gro60a,Gro60b} we adopt a Tannakian approach to our definition of a fundamental group. In fact, in the case of a differentiable manifold without boundary this becomes very concrete given the classical result that there is an equivalence between the category of representations of the fundamental group and the category of flat connections on vector bundles over that manifold (see for instance \cite[Proposition I.2.5]{Kob87}). One may then reconstruct (the pro-algebraic completion of) the fundamental group of that manifold as the automorphism group of the corresponding Tannakian category ({\em cf.} \cite{Sav72} and \cite{Del90}). 

The noncommutative generalisation of the fundamental group that we propose here is based on the construction in Section \ref{sect:conn-nc} of a category of finitely generated projective bimodules over a differential graded algebra (dga) together with a connection. The difference with previous approaches to connections on bimodules (such as for instance \cite{DuboisVioletteGDA}) is that we demand compatibility of the bimodule connection with {\em all} elements in the dga, which also solves the apparent incompatibility recorded in \cite[Example 2.13]{KLS08}.

Our main result (Section \ref{sect:fund-grp}) is then that under some analytical assumptions on the dga the subcategory where the connections are flat is a Tannakian category. These assumptions are completely natural in the context of noncommutative differential geometry and are for instance fulfilled for quantum metric spaces in the sense of Rieffel \cite{compactquantummetricspaces}. It leads us to define the fundamental group $\pi^1$ of the dga as the group corresponding to this Tannakian category. We establish the appropriate functorial properties of the fundamental group, as well as homotopy and Morita invariance in Section \ref{sect:props}. Note that an alternative way to arrive at a fundamental group for noncommutative spaces would be to develop a theory of noncommutative covering spaces and consider the corresponding automorphisms. This is for instance the approach taken in \cite{Canlubo,Ivankov,SW17}. 

In Section \ref{sect:ex} we then illustrate our construction by considering examples, including all noncommutative tori and, more generally, we consider toric noncommutative manifolds. 

\medskip

\subsubsection*{Acknowledgements}
We would like to thank Ben Moonen and Steffen Sagave for inspiring discussions and valuable suggestions, as well as Alain Connes for several useful remarks. 

\subsubsection*{Notation}
Algebras (assumed to be unital and over $\mathbb C$): $\A,\B$, {\em etc.};

$C^*$-algebras: $A$;

Graded algebras: $\Omega \A = \oplus_{k \geq 0} \Omega^k \A$ with $\Omega^0 \A \equiv \A$ and $|\alpha|$ the degree of $\alpha \in \Omega \A$; 

Differential graded algebras $(\Omega \A, d)$;

Graded modules over graded algebras: $\Omega \E, \Omega \F, \Omega \G$, {\em, etc.}.

Graded centre: $Z_g(\Omega \A)$, $Z_g(\A)\equiv Z_g(\Omega\A)^0$, $Z_g(\E)$, {\em etc.}.

\section{Categories of connections}
\label{sect:conn-nc}
In this section we will define connections over a noncommutative space. These noncommutative spaces are described by differential graded algebras and, in fact, the results in this section apply to any (noncommutative) dga. We consider the category of all connections on finitely generated modules for a dga $(\Omega \A,d)$, and we will prove that it is a rigid tensor category. Of course, the same then holds for the subcategory of flat connections. We also show that the category only depends on the (graded commutative) centre of the dga.

Some words on notation: in this paper we will write a dga as $(\Omega \A,d)$. Even though this might be considered unnatural as dga's are usually written simply as $(\A,d)$, we use this notation to stress that it furnishes a differential geometric structure on the noncommutative space described by $\A$ (and a $C^*$-algebra $A$). This becomes even more apparent in Section \ref{sect:fund-grp} where we impose some analytical conditions on $A$. For the same reason, we will write $(\Omega \E,\nabla)$ for a graded bimodule with a connection, which in the flat case is then of course nothing but a differential graded bimodule. 

\subsection{Connections and bimodules over a dga}
There are multiple ways to generalise vector bundles to a noncommutative space. By Serre--Swan Theorem the vector bundles over a manifold correspond to finitely generated projective modules over the algebra $C^\infty(M)$ of smooth functions. Over a noncommutative algebra we can look at right modules, but there is no suitable tensor product of right modules. We can look at bimodules, but here we have to remember that over a commutative space, only very special bimodules are allowed: the multiplication on the left is the same as on the right. We can look at bimodules that satisfy this condition for the centre of the algebra. Instead we look at a more restrictive class of bimodules, namely those that are a direct summand of a free finite bimodule (\cite{Dubois-Violette-derivations-et-calcul} defines \emph{diagonal bimodules} as modules that are a summand of a free module).

This is very natural when one considers bimodule connections. Namely, connections are usually defined as maps $ \E \to \E \tensor \Omega^1 \A$ that satisfy the appropriate Leibniz rules. However to satisfy the left Leibniz rule the image should actually be $\Omega^1 \A \tensor \E$. This is solved in \cite{DuboisVioletteGDA} by using an isomorphism $\sigma: \E\tensor \Omega^1 \A\to\Omega^1 \A\tensor \E$. However, then the tensor product of two flat connections is not necessarily flat. Instead if one considers graded bimodules $\Omega \E = \oplus_{k \geq 0} \Omega^k \E$ over a graded algebra $\Omega \A = \oplus_{k \geq 0} \Omega^k \A$ that are a summand of a free finite module, we automatically get isomorphisms $\Omega^1 \E\cong \E\tensor\Omega^1 \A\cong \Omega^1 \A\tensor \E$ (see Lemma \ref{lem-omegakE}). Moreover, now the tensor product of flat connections is always flat (see Lemma \ref{lem-curvature_tensor_product}).

Let us now proceed and describe the precise algebraic setup. 
\begin{definition}
A \emph{differential graded algebra} or dga is a graded algebra $\Omega \A$ together with a $\mathbb C$-linear map $d:\Omega \A\to\Omega \A$ of degree +1 satisfying the Leibniz rule
\[d(\omega\nu)=d\omega\cdot\nu+(-1)^{|\omega|}\omega d\nu\]
for $\omega,\nu\in\Omega \A$, and $d^2=0$.
\end{definition}




\begin{notation}
Let $\Omega \E$ be a graded $\Omega \A$-bimodule. We define the graded commutator $[\cdot,\cdot]:\Omega \A\times \Omega \E\to \Omega \E$ as $[\alpha,\epsilon]=\alpha\epsilon-(-1)^{|\alpha|\cdot|\epsilon|}\epsilon\alpha$ for $\alpha\in\Omega \A,\epsilon\in\Omega \E$ and $\mathbb C$-bilinearly extended to non-homogeneous elements. We use the same notation for the graded commutator $[\cdot,\cdot]:\Omega \A\times \Omega \A\to \Omega \A$ defined similarly.
\end{notation}

\begin{definition}
Let $\Omega \A$ be any dga. We call a graded $\Omega \A$-bimodule $\Omega \E$ \emph{finitely generated projective} (fgp) if there is another graded $\Omega \A$-bimodule $\Omega \F$ satisfying $\Omega \E\oplus \Omega \F=\Omega \A^{ n}$ for some integer $n$, as graded $\Omega \A$-bimodules. 
\end{definition}
The following lemma shows why it is convenient to work with these graded fgp $\Omega \A$-bimodules instead of just fgp bimodules over $\A$.

\begin{lemma}\label{lem-omegakE}
  Let $\Omega \E$ be a graded fgp $\Omega \A$-bimodule. Write $\Omega^0 \A=\A$ and $\Omega^0 \E= \E$.
  Then for all $k\geq0$ the multiplication induces isomorphisms
  $$
   \E\tensor_{ \A}\Omega^k \A\xrightarrow\sim\Omega^k \E,\qquad 
    \Omega^k \A\tensor_{ \A}\E\xrightarrow\sim\Omega^k \E.
$$
\end{lemma}
\begin{proof}
Let $\Omega \F$ be another graded $\Omega \A$-bimodule satisfying $\Omega \E\oplus \Omega \F\cong \Omega \A^{ n}$. Then we have the following commuting diagram:
\[\begin{tikzcd}
\E\tensor_{ \A}\Omega^k \A\oplus \F \tensor_{\A}\Omega^k \A  \arrow[d,"\sim"]\arrow[rr] &  &  \Omega^k \E\oplus\Omega^k \F \arrow[d,"\sim"]\\
( \E\oplus \F)\tensor_{ \A}\Omega^k \A \arrow[r,"\sim"]&  \A^{ n}\tensor_{ \A}\Omega^k \A \arrow[r,"\sim"] & (\Omega^k \A)^{ n}.
\end{tikzcd}\]
Here the top arrow is the direct sum of the maps $ \E \tensor _{\A}\Omega^k \A\to \Omega^k \E$ and $ \F\tensor_{\A}\Omega^k \A\to \Omega^k \F$. All the other maps are isomorphisms, so these maps are isomorphisms as well. This shows that the first map in the lemma is an isomorphism and the second one follows analogously.
\end{proof}

From this we see that if $\Omega \A$ is graded commutative, then $\Omega \E$ is completely determined by just $\E$: it is the module $\Omega \A\tensor_{\A} \E$.

\begin{definition}
Let $\Omega \E$ be a graded fgp $\Omega \A$-bimodule. A \emph{connection} on $\Omega \E$ is a $\mathbb C$-linear map
\[\nabla:\Omega \E\to \Omega \E\]
of degree +1 satisfying the following equations for $\epsilon\in\Omega \E,\alpha\in \Omega \A$:
\begin{align}
  \label{eq:right-leibniz}
    \nabla(\epsilon\alpha)&=\nabla(\epsilon)\alpha+(-1)^{|\epsilon|}\epsilon d\alpha,\\
    \nabla(\alpha\epsilon)&=d\alpha\cdot\epsilon+(-1)^{|\alpha|}\alpha\nabla(\epsilon).
\end{align}
\end{definition}
\begin{remark}
The equations in this definition are called the right Leibniz rule and the left Leibniz rule, respectively. They are meaningful because the elements of $\Omega \E$ can be multiplied with elements of $\Omega \A$ both on the left and the right.
\end{remark}
\begin{definition}
Let $\Omega \E$ be a graded fgp $\Omega \A$-bimodule with a connection $\nabla$. The curvature of $\nabla$ is the map $\nabla^2:\Omega \E\to\Omega \E$ of degree +2.
\end{definition}
\begin{remark}
The curvature is an $\Omega \A$-bilinear map: for $\epsilon\in \Omega \E$ and $\alpha\in\Omega \A$ we have
\begin{align*}
\nabla^2(\epsilon\alpha)&=\nabla(\nabla(\epsilon)\alpha+(-1)^{|\epsilon|}\epsilon d\alpha)\\
&=\nabla^2(\epsilon)\alpha+(-1)^{|\epsilon|+1}\nabla(\epsilon)d\alpha+(-1)^{|\epsilon|}\nabla(\epsilon)d\alpha+\epsilon d^2\alpha\\
&=\nabla^2(\epsilon)\alpha
\end{align*}
and
\begin{align*}
\nabla^2(\alpha\epsilon)&=\nabla(d\alpha\cdot\epsilon+(-1)^{|\alpha|}\alpha\nabla(\epsilon))\\
&=d^2\alpha\cdot\epsilon+(-1)^{|\alpha|+1}d\alpha\nabla(\epsilon)+(-1)^{|\alpha|}d\alpha\nabla(\epsilon)+\alpha\nabla^2(\epsilon)\\
&=\alpha\nabla^2(\epsilon).
\end{align*}
\end{remark}
\begin{definition}
A connection $\nabla$ is called \emph{flat} if the curvature $\nabla^2$ is zero.
\end{definition}
Now we can define the category of flat connections.
\begin{definition}
We define $\C(\Omega \A)$ to be the category whose objects are graded fgp $\Omega \A$-bimodules with a connection, and whose morphisms are graded fgp $\Omega \A$-bimodule morphisms of degree 0 which commute with the connections. Let $\C_{\flat}(\Omega \A)$ be the full subcategory where the connections are required to be flat.
\end{definition}
\begin{example}
  \label{ex:manifold}
Let $M$ be a manifold without boundary. The corresponding dga is the de Rham differential algebra of the manifold $\Omega \A=\Omega M$. Any fgp $\Omega \A$-bimodule $\Omega \E$ is determined by the fgp $\A$-bimodule $\E$ by $\Omega \E=\Omega \A \tensor_{ \A} \E$. This corresponds to a vector bundle over $M$ by the Serre--Swan Theorem. A flat connection on $\E$ corresponds to a (usual) flat connection on this vector bundle. So $\C(\Omega M)$ is equivalent to the category of vector bundles over $M$ with a flat connection, which in turn is equivalent to the category of representations of $\pi_1(M)$ by \cite[Proposition I.2.5]{Kob87}.
\end{example}

\subsection{Connections over the graded centre}
We will now show that each graded fgp $\Omega \A$-bimodule is determined by a graded fgp bimodule over a graded commutative subalgebra. For this we need the following definitions:
\begin{definition}
Let $\Omega \A$ be a dga. We define the graded commutative centre $Z_g(\Omega \A)$ as
\[Z_g(\Omega \A)=\{\alpha\in \Omega \A\mid [\alpha,\nu]=0\text{ for all }\nu\in\Omega \A\}.\]
If $\Omega \E$ is an fgp $\Omega \A$-bimodule we define $Z_g(\Omega \E)$ as
\[Z_g(\Omega \E)=\{\epsilon\in \Omega \E\mid [\epsilon,\nu]=0\text{ for all }\nu\in\Omega \A\}.\]
\end{definition}
\begin{lemma}\label{lem-graded_centre}
With notations as in the previous definition we have the following, part of which is shown in \cite{DuboisVioletteGDA}:
\begin{enumerate}[(i)]
    \item The graded commutative centre of the algebra $Z_g(\Omega \A)$ is a dga.
    \item The graded commutative centre of the module $Z_g(\Omega \E)$ is a graded fgp bimodule over $Z_g(\Omega \A)$.
    \item Multiplication gives an isomorphism of graded bimodules $Z_g(\Omega \E)\tensor_{Z_g(\Omega \A)}\Omega \A\xrightarrow\sim\Omega \E$.
\end{enumerate}
\end{lemma}
\begin{proof}
\begin{enumerate}[(i)]
    \item An easy calculation shows that $Z_g(\Omega \A)$ is a subalgebra of $\Omega \A$, using that $[\alpha\beta,\nu]=\alpha[\beta,\nu]+(-1)^{|\beta|\cdot|\nu|}[\alpha,\nu]\beta$ for $\alpha,\beta,\nu\in\Omega \A$. It is closed under $d$ because $[d\alpha,\nu]=d[\alpha,\nu]-(-1)^{|\alpha|}[\alpha,d\nu]$ for $\alpha,\nu\in\Omega \A$.
    \item An easy calculation shows that $Z_g(\Omega \E)$ is a graded bimodule over $Z_g(\Omega \A)$, using $[\alpha\epsilon,\nu]=\alpha[\epsilon,\nu]+(-1)^{|\epsilon|\cdot|\nu|}[\alpha,\nu]\epsilon$ for $\alpha,\nu\in\Omega \A,\epsilon\in\Omega \E$. If $\Omega \E\oplus \Omega \F=\Omega \A^{ n}$ it follows directly that $Z_g(\Omega \E)\oplus Z_g(\Omega \F)=(Z_g(\Omega \A))^{ n}$, so $Z_g(\Omega \E)$ is a graded fgp $Z_g(\Omega \A)$-bimodule.
    \item Let $\Omega \F$ be another graded fgp $\Omega \A$-bimodule satisfying $\Omega \E\oplus\Omega \F=\Omega \A^{ n}$. Then we have the following commuting diagram:
    \[\begin{tikzcd}[column sep=small]
    Z_g(\Omega \E)\tensor_{Z_g(\Omega \A)}\Omega \A\oplus Z_g(\Omega \F)\tensor_{Z_g(\Omega \A)}\Omega \A \arrow[r]\arrow[d,"\sim"]& \Omega \E\oplus \Omega \F \arrow[d,"\sim"] \\
     (Z_g(\Omega \E)\oplus Z_g(\Omega \F))\tensor_{Z_g(\Omega \A)}\Omega \A \arrow[d,"\sim"]&\Omega \A^{ n}\\
   Z_g(\Omega \E\oplus \Omega \F)\tensor_{Z_g(\Omega \A)}\Omega \A \arrow[r,"\sim"] & Z_g(\Omega \A^{ n})\tensor_{Z_g(\Omega \A)}\Omega \A. \arrow[u,"\sim"]
    \end{tikzcd}\]
    Here the top arrow is the direct sum of the maps $Z_g(\Omega \E)\tensor_{Z_g(\Omega \A)}\Omega \A\to \Omega \E$ and $Z_g(\Omega \F)\tensor_{Z_g(\Omega \A)}\Omega \A\to \Omega \F$ induced by multiplication. All other arrows in the diagram are isomorphisms, so these maps are isomorphisms as well.\qedhere
\end{enumerate}
\end{proof}
So all graded fgp $\Omega \A$-bimodules are determined by a graded fgp $Z_g(\Omega \A)$-bimodule. This is in turn determined by an fgp $Z_g(\Omega \A)$-module. Note that $Z_g(\A)$ may be smaller than the centre of the algebra $\A$.

\begin{theorem}\label{thm-reduction_to_commutative_case}
We have an equivalence of categories:
\[\C(\Omega \A)\xrightarrow\sim \C(Z_g(\Omega \A)),\]
sending an object $(\Omega \E,\nabla)$ to $Z_g(\Omega \E)$ with the restriction of $\nabla$. This equivalence restricts to an equivalence $\C_{\flat}(\Omega \A)\xrightarrow\sim \C_{\flat}(Z_g(\Omega \A))$.
\end{theorem}
\begin{proof}
Let $(\Omega \E,\nabla)$ be an object of $\C(\Omega \A)$. For $\epsilon \in Z_g(\Omega \E)$ and $\alpha\in \Omega \A$ we have
\[\nabla(\epsilon\alpha)=\nabla(\epsilon)\alpha+(-1)^{|\epsilon|}\epsilon d\alpha\]
but also
\begin{align*}
    \nabla(\epsilon\alpha)&=(-1)^{|\epsilon|\cdot|\alpha|}\nabla(\alpha\epsilon)\\
    &=(-1)^{(|\epsilon|+1)|\alpha|}\alpha\nabla(\epsilon)+(-1)^{|\epsilon|\cdot|\alpha|}d\alpha\cdot \epsilon\\
    &=(-1)^{(|\epsilon|+1)|\alpha|}\alpha\nabla(\epsilon)+(-1)^{|\epsilon|}\epsilon d\alpha.
\end{align*}
So we get
\[\nabla(\epsilon)\alpha=(-1)^{(|\epsilon|+1)|\alpha|}\alpha\nabla(\epsilon).\]
Since this holds for all $\alpha\in \Omega \A$ we conclude that $\nabla(\epsilon)\in Z_g(\Omega \E)$. So $\nabla$ restricts to a function $Z_g(\Omega \E)\to Z_g(\Omega \E)$. Then $(Z_g(\Omega \E),\nabla_{|Z_g(\Omega \E)})$ is an object of $\C(Z_g(\Omega \A))$. It is easy to see that this is a functorial construction.

Conversely, let $(\Omega \F,\nabla)$ be an object of $\C(Z_g(\Omega \A))$. Then we can define the graded fgp $\Omega \A$-bimodule $\Omega \F\tensor_{Z_g(\Omega \A)}\Omega \A$, and the connection $\tilde\nabla$ given by
\[\tilde\nabla(\zeta\tensor\alpha)=\nabla(\zeta)\tensor \alpha+(-1)^{|\zeta|}\zeta\tensor d(\alpha)\]
for $\zeta\in \Omega \F,\alpha\in \Omega \A$. This gives an object of $\C(\Omega \A)$. It is then easy to show that this construction is also functorial, and that the two functors thus defined are inverse to each other.

Finally, we observe that the functor and its inverse preserve flatness.
\end{proof}

\begin{remark}
  \label{rem:reduce-to-0}
For a graded commutative dga $\Omega \A$ we have a different way to describe the category $\C(\Omega \A)$. Remember that for a graded fgp bimodule $\Omega \E$ we have the isomorphisms $\Omega \E=\E\tensor_\A\Omega \A$. The restriction of $\nabla$ to $\E$ is then a map $\nabla_0: \E\to \E\tensor_ \A\Omega^1 \A$, satisfying $\nabla_0(ea)=\nabla_0(e)a+e\tensor da$. Conversely each such $\nabla_0$ may be extended to $\nabla:\Omega \E\to\Omega \E$ by setting $\nabla(e\tensor \omega)=\nabla_0(e)\omega+e\tensor d\omega$ for $e\in\Omega \E,\omega\in\Omega \A$. So we can describe $\C(\Omega \A)$ as the category with objects fgp $ \A$-modules $ \E$ with a connection $\nabla_0: \E\to \E\tensor_{ \A}\Omega^1 \A$. 
\end{remark}

\begin{example}\label{ex-two_points_at_finite_distance}
  We consider the dga corresponding to the noncommutative space of two points at finite distance (see \cite{CL91} or \cite[pp 116--118]{Landi}). If is given by
  $$
  \Omega^{2k}  \A= 
  \left\{\begin{pmatrix}\alpha&0\\0&\delta\end{pmatrix}\in M_2(\mathbb C)\right\} ; \qquad
    \Omega^{2k+1} \A= 
      \left\{\begin{pmatrix}0&\beta\\ \gamma&0\end{pmatrix}\in M_2(\mathbb C)\right\} 
        $$
        for all $k \geq 0$. The differential is given by 
        $$
        \begin{aligned}
          d:\Omega^{2k} \A&\to \Omega^{2k+1} \A\\
          \begin{pmatrix}\alpha&0\\0&\delta\end{pmatrix}&\mapsto \begin{pmatrix}0&\delta-\alpha\\ \alpha-\delta&0\end{pmatrix}\end{aligned},\qquad 
        \begin{aligned} 
          d:\Omega^{2k+1} \A &\to \Omega^{2k+2} \A\\
          \begin{pmatrix}0&\beta\\ \gamma&0\end{pmatrix}&\mapsto \begin{pmatrix}\beta+\gamma & 0\\ 0&\beta+\gamma\end{pmatrix}.
        \end{aligned}
        $$
We have $Z_g(\Omega ^{2k}\A)=\mathbb C$ and $Z_g(\Omega^{2k+1}\A)=0$ so that in view of Remark \ref{rem:reduce-to-0} the fgp bimodules over $Z_g(\Omega \A)$ are simply determined by a vector space over $\mathbb C$. We conclude that $\C(\Omega \A)\cong \C(Z_g(\Omega \A))$ is equivalent to the category of vector spaces.
\end{example}

\subsection{Tensor products}
In this subsection we will construct the tensor product for the category $\C(\Omega \A)$. 
\begin{proposition}
Let $(\Omega \E,\nabla^\E)$ and $(\Omega \F,\nabla^\F)$ be objects of $\C(\Omega \A)$. Then $\Omega\G=\Omega \E\tensor_{\Omega \A}\Omega \F$ has the structure of an fgp graded bimodule over $\Omega \A$, and we can construct a connection $\nabla^\G$ satisfying
\begin{equation}
  \label{eq:tensor-conn}
  \nabla^\G(\epsilon\tensor\zeta)=\nabla^\E(\epsilon)\tensor\zeta+(-1)^{|\epsilon|}\epsilon\tensor\nabla^\F(\zeta)
  \end{equation}
for $\epsilon\in\Omega \E,\zeta\in\Omega \F$.
\end{proposition}
\begin{proof}
  The left action of $\Omega \A$ on $\Omega \E$ and the right action of $\Omega \A$ on $\Omega \F$ make $\Omega\G$ into a $\Omega \A$-bimodule. A grading on the tensor product is given as follows: for any $r\geq 0$ the degree $r$ subspace $\Omega^r\G$ is the linear span of elements $\epsilon\tensor\zeta$, with $\epsilon\in\Omega^k \E,\zeta\in \Omega ^l\F,k+l=r$. To show that $\Omega\G$ is fgp, suppose that $\Omega \E\oplus \Omega \E'\cong \Omega \A^{ m}$ and $\Omega \F\oplus\Omega \F'\cong\Omega \A^{ n}$. Then
 \[\Omega\G\oplus\Omega \E'\tensor_{\Omega \A}\Omega \F\oplus (\Omega \F')^{ m}\cong (\Omega \F)^{ m} \oplus(\Omega \F')^{ m}\cong (\Omega \A)^{ mn}.\]
So $\Omega\G$ is a graded fgp $\Omega \A$-bimodule.

We can then define the connection $\nabla^\G$ by Equation \eqref{eq:tensor-conn} on pure tensors, and extend $\mathbb C$-linearly.
To show that it is well-defined, let $\alpha\in \Omega \A$. Then by the above definition we have
\begin{align*}
\nabla^\G(\epsilon\alpha\tensor\zeta)&=\nabla^\E(\epsilon\alpha)\tensor\zeta+(-1)^{|\epsilon|+|\alpha|}\epsilon\alpha\tensor\nabla^\F(\zeta)\\
&=\nabla^\E(\epsilon)\alpha\tensor\zeta+(-1)^{|\epsilon|}\epsilon d\alpha\tensor\zeta+(-1)^{|\epsilon|+|\alpha|}\epsilon\alpha\tensor\nabla^\F(\zeta)
\end{align*}while
\begin{align*}
\nabla^\G(\epsilon\tensor\alpha\zeta)&=\nabla^\E(\epsilon)\tensor\alpha\zeta+(-1)^{|\epsilon|}\epsilon\tensor\nabla^\F(\alpha\zeta)\\
&=\nabla^\E(\epsilon)\tensor\alpha\zeta+(-1)^{|\epsilon|}\epsilon\tensor d\alpha\cdot \zeta+(-1)^{|\epsilon|+|\alpha|}\epsilon\tensor\alpha\nabla^\F(\zeta).
\end{align*}
Since these are the same, $\nabla^\G$ is well-defined.

Lastly, $\nabla^\G$ satisfies the Leibniz rules: for $\epsilon\in \Omega \E,\zeta\in\Omega \F,\alpha\in\Omega \A$ we have
\begin{align*}
\nabla^\G(\alpha\epsilon\tensor\zeta)&=\nabla^\E(\alpha\epsilon)\tensor\zeta+(-1)^{|\epsilon|+|\alpha|}\alpha\epsilon\tensor\nabla^\F(\zeta)\\
&=d\alpha\cdot\epsilon\tensor\zeta+(-1)^{|\alpha|}\alpha\nabla^\E(\epsilon)\tensor\zeta+(-1)^{|\epsilon|+|\alpha|}\alpha\epsilon\tensor\nabla^\F(\zeta)\\
&=d\alpha\cdot\epsilon\tensor\zeta+(-1)^{|\alpha|}\alpha\nabla^\G(\epsilon\tensor \zeta)
\end{align*}
and
\begin{align*}
\nabla^\G(\epsilon\tensor\zeta\alpha)&=\nabla^\E(\epsilon)\tensor\zeta\alpha+(-1)^{|\epsilon|}\epsilon\tensor\nabla^\G(\zeta\alpha)\\
&=\nabla^\E(\epsilon)\tensor\zeta\alpha+(-1)^{|\epsilon|}\epsilon\tensor\nabla^\G(\zeta)\alpha+(-1)^{|\epsilon|+|\zeta|}\epsilon\tensor\zeta d\alpha\\
&=\nabla^\G(\epsilon\tensor\zeta)\alpha+(-1)^{|\epsilon|+|\zeta|}\epsilon\tensor\zeta d\alpha.\qedhere
\end{align*}
\end{proof}
The curvature on the tensor product is easily calculated:
\begin{lemma}\label{lem-curvature_tensor_product}
In the notation of the previous lemma, we have
\[(\nabla^\G)^2=(\nabla^\E)^2\tensor\Omega \F\oplus\Omega \E\tensor(\nabla^\F)^2.\]
In particular, the tensor product of flat connections is again flat.
\end{lemma}
\begin{proof}
For $\epsilon\in\Omega \E,\zeta\in\Omega \F$ we have
\begin{align*}
&    (\nabla^\G)^2(\epsilon\tensor\zeta)=\nabla^\G(\nabla^\E(\epsilon)\tensor\zeta+(-1)^{|\zeta|}\epsilon\tensor\nabla^\F(\zeta))\\
    &\qquad =(\nabla^\E)^2(\epsilon)\tensor\zeta+(-1)^{|\zeta|+1}\nabla^\E(\epsilon)\tensor\nabla^\F(\zeta)+(-1)^{|\zeta|}\nabla^\E(\epsilon)\tensor\nabla^\F(\zeta)+\epsilon\tensor(\nabla^\F)^2(\zeta)\\
    & \qquad =(\nabla^\E)^2(\epsilon)\tensor\zeta+\epsilon\tensor(\nabla^\F)^2(\zeta).\qedhere
\end{align*}
\end{proof}
\begin{remark}
The above lemma does not apply for some other definitions of connections on bimodules, see for instance \cite[Example 2.13]{KLS08}.
\end{remark}

It is easy to see that this tensor product is associative. The tensor product commutes with the equivalence of categories $\C(\Omega \A)\to \C(Z_g(\Omega \A))$ from Theorem \ref{thm-reduction_to_commutative_case}. In the commutative case it is easy to see that the tensor product is also commutative; so we have a commutativity constraint in the general case as well.

There is a unit in $\C(\Omega \A)$: it is the bimodule $\Omega \A$ with the connection $d$. It is easy to see that the isomorphism $ \Omega \E\to \Omega \E\tensor_{\Omega \A}\Omega \A$ intertwines the connection $\nabla^\E$ with the tensor product connection on $\Omega \E\tensor_{\Omega \A} \Omega \A$.

This makes $(\C(\Omega \A),\tensor)$ into a tensor category, and the same applies to the subcategory $(\C_{\flat}(\Omega \A),\otimes)$.

\subsection{Duals}
We will now construct dual objects in the category $\C(\Omega \A)$.

\begin{proposition}
Let $(\Omega \E,\nabla)$ be an object of $\C(\Omega \A)$. Then the dual module $\Omega \E^\vee=\Hom_{\Omega \A}(\Omega \E,\Omega \A)$ of right-$\Omega \A$-linear maps from $\Omega \E$ to $\Omega \A$ is an fgp $\Omega \A$-bimodule. Moreover, there is a connection on $\Omega \E^\vee$ satisfying
\begin{equation}
  \label{eq:dual-conn}
  \langle\nabla^\vee(\theta),\epsilon\rangle=d\langle\theta,\epsilon\rangle-(-1)^{|\theta|}\langle\theta,\nabla\epsilon\rangle
  \end{equation}
for $\theta\in\Omega \E^\vee,\epsilon\in\Omega \E$. Here the angled brackets denote the pairing between $\Omega \E^\vee$ and $\Omega \E$.
\end{proposition}
\begin{proof}
  The bimodule structure is given by $\langle\theta\alpha,\epsilon\rangle=\langle\theta,\alpha\epsilon\rangle$ and $\langle\alpha\theta,\epsilon\rangle=\alpha\langle\theta,\epsilon\rangle$ for $\theta\in\Omega \E^\vee,\alpha\in\Omega \A,\epsilon\in\Omega \E$. There is a natural grading on $\Omega \E^\vee$ where in degree $k$ we find 
  the homogeneous maps of degree $k$. If $\Omega \E\oplus\Omega \F\cong \Omega \A^{ n}$ we get $\Omega \E^\vee\oplus\Omega \F^\vee=(\Omega \E\oplus\Omega \F)^\vee\cong \Omega \A^{ n}$. So $\Omega \E^\vee$ is a graded fgp $\Omega \A$-bimodule.

We can define the connection $\nabla^\vee$ on $\Omega \E^\vee$ by Equation \eqref{eq:dual-conn}.
This is well-defined because $\langle\nabla^\vee(\theta),-\rangle$ is indeed a right-linear map with this definition: for $\theta\in \Omega \E^\vee,\epsilon\in\Omega \E,\alpha\in\Omega \A$:
\begin{align*}
\langle\nabla^\vee(\theta),\epsilon\alpha\rangle&=d\langle\theta,\epsilon\alpha\rangle-(-1)^{|\theta|}\langle\theta,\nabla(\epsilon\alpha)\rangle\\
&=d(\langle\theta,\epsilon\rangle\alpha)-(-1)^{|\theta|}\langle\theta,\nabla(\epsilon)\alpha+(-1)^{|\theta|}\epsilon d\alpha\rangle\\
&=d\langle\theta,\epsilon\rangle\alpha+(-1)^{|\theta|+|\epsilon|}\langle\theta,\epsilon\rangle d\alpha-(-1)^{|\theta|}\langle\theta,\nabla(\epsilon)\rangle\alpha-(-1)^{|\theta|+|\epsilon|}\langle\theta,\epsilon\rangle d\alpha\\
&=\langle\nabla^\vee(\theta),\epsilon\rangle\alpha.
\end{align*}This satisfies the Leibniz rules: for $\theta\in\Omega \E^\vee,\epsilon\in\Omega \E,\alpha\in\Omega \A$ we have
\begin{align*}
\langle\nabla^\vee(\theta\alpha),\epsilon\rangle&=d\langle\theta\alpha,\epsilon\rangle-(-1)^{|\theta|+|\alpha|}\langle\theta\alpha,\nabla(\epsilon)\rangle\\
&=d\langle\theta,\alpha\epsilon\rangle-(-1)^{|\theta|+|\alpha|}\langle\theta,\alpha\nabla(\epsilon)\rangle\\
&=d\langle\theta,\alpha\epsilon\rangle-(-1)^{|\theta|+|\alpha|}\langle\theta,\nabla(\alpha\epsilon)\rangle+(-1)^{|\theta|}\langle\theta,d\alpha\cdot\epsilon\rangle\\
&=\langle\nabla^\vee(\theta),\alpha\epsilon\rangle+(-1)^{|\theta|}\langle\theta d\alpha,\epsilon\rangle\\
&=\langle\nabla^\vee(\theta)\alpha+(-1)^{|\theta|}\theta d\alpha,\epsilon\rangle
\end{align*}
and
\begin{align*}
\langle\nabla^\vee(\alpha\theta),\epsilon\rangle&=d\langle\alpha\theta,\epsilon\rangle-(-1)^{|\theta|+|\alpha|}\langle\alpha\theta,\nabla(\epsilon)\rangle\\
&=d\alpha\langle\theta,\epsilon\rangle+(-1)^{|\alpha|}\alpha d\langle\theta,\epsilon\rangle-(-1)^{|\theta|+|\alpha|}\alpha\langle\theta,\nabla(\epsilon)\rangle\\
&=\langle d\alpha\cdot\theta+(-1)^{|\alpha|}\alpha\nabla^\vee(\theta),\epsilon\rangle.\qedhere
\end{align*}\end{proof}

We can compute the curvature of the dual connection:
\begin{lemma}
In the notation of the previous lemma, the curvature of $\nabla^\vee$ is minus the dual of the curvature of $\nabla$, that is, for $\theta\in\Omega \E^\vee,\epsilon\in\Omega \E$ we have
\[\langle(\nabla^\vee)^2(\theta),\epsilon\rangle=-\langle\theta,\nabla^2(\epsilon)\rangle.\]
In particular, the dual of a flat connection is again flat.
\end{lemma}
\begin{proof}
We have for $\theta\in\Omega \E^\vee,\epsilon\in\Omega \E$:
\begin{align*}
    \langle((\nabla^\vee)^2(\theta),\epsilon\rangle&=d\langle\nabla^\vee(\theta),\epsilon\rangle-(-1)^{|\theta|+1}\langle\nabla^\vee(\theta),\nabla(\epsilon)\rangle\\
    &=d(d\langle\theta,\epsilon\rangle-(-1)^{|\theta|}\langle\theta,\nabla(\epsilon)\rangle)-(-1)^{|\theta|+1}d\langle\theta,\nabla(\epsilon)\rangle-\langle\theta,\nabla^2(\epsilon)\rangle\\
    &=-\langle\theta,\nabla^2(\epsilon)\rangle.\qedhere
\end{align*}
\end{proof}
Writing $\Omega \E=Z_g(\Omega \E)\tensor_{Z_g(\Omega \A)}\Omega \A$ we have
\[\Omega \E^\vee=\Hom_{\Omega \A}(Z_g(\Omega \E)\tensor_{Z_g(\Omega \A)}\Omega \A,\Omega \A)=Z_g(\Omega \E)^\vee\tensor_{Z_g(\Omega \A)}\Omega \A.\]
So the equivalence of categories $\C(\Omega \A)\to\C(Z_g(\Omega \A))$ commutes with the taking of duals. In particular this shows that the dual $\Omega \E^\vee=\Hom_{\Omega \A}(\Omega \E,\Omega \A)$ is naturally isomorphic to the space of left-linear functions $_{\Omega \A}\Hom(\Omega \E,\Omega \A)$.

Since $\Omega \E$ is an fgp bimodule, we have for each graded fgp bimodule $\Omega \F$ an isomorphism $\Hom(\Omega \F\tensor\Omega \E,\Omega \A)=\Hom(\Omega \F,\Omega \E^\vee\tensor\Omega \A)$. An easy calculation shows that this isomorphism continues to hold for morphisms that commute with connections, if connections on $\Omega \E$ and $\Omega \F$ are given. So $(\Omega \E^\vee,\nabla^\vee)$ is a dual object to $(\Omega \E,\nabla)$. The morphism $\Omega \E\to(\Omega \E^\vee)^\vee$ is an isomorphism because $\Omega \E$ is fgp. So every object is reflexive, and $(\C(\Omega \A),\tensor)$ is a rigid tensor category, and the subcategory $(\C_{\flat}(\Omega \A), \tensor)$ is also a rigid tensor category.

\section{A Tannakian category and the fundamental group}
\label{sect:fund-grp}
In this section we will show that under some analytical conditions on the dga the category $\C_{\flat}(\Omega \A)$ constructed above is actually a neutral Tannakian category. In particular, if the algebra $\A= \Omega^0\A$ is a dense $*$-subalgebra of a quantum metric space in the sense of Rieffel \cite{compactquantummetricspaces} these conditions are met, so that our results apply to a broad class of noncommutative differential spaces. We then define the fundamental group of the pertinent noncommutative space as the automorphism group of the fiber functor in this Tannakian category (we refer to \cite{Del90} for more details). 

\subsection{Abelianness of the category}
In this subsection we will study when the category $\C(\Omega \A)$ is abelian. Using Theorem \ref{thm-reduction_to_commutative_case} we can always reduce to a graded commutative dga and this allows for the description of $\C(\Omega\A)$ simply as the category of fgp $\A$-modules equipped with a connection (cf. Remark \ref{rem:reduce-to-0}). 
We will assume that $\A$ is a unital $*$-algebra that is dense in a unital $C^*$-algebra $A$. This will be necessary for some of the constructions below. Moreover we also need the star operation on $\Omega \A$.
\begin{definition}
A $*$-dga is a dga $\Omega \A$ with a linear involution $*$, satisfying $(\alpha\beta)^*=\beta^*\alpha^*$ and $d(\alpha^*)=d(\alpha)^*$.
\end{definition}

We also assume that the elements in $ \A$ that are invertible in $A$ are also invertible in $\A$, so $ \A \cap A^\times=\A^\times$. This is in particular the case if $ \A$ is stable under holomorphic functional calculus (see \cite[p.134]{GraciaBondia}).

The category $\C(\Omega \A)$ is always an additive category: given two objects $(\Omega \E,\nabla^\E)$ and $(\Omega \F,\nabla^\F)$ the morphisms from $\Omega \E$ to $\Omega \F$ form an additive group, and there is an object $\Omega \E\oplus\Omega \F$ where the connection is simply given by $\nabla^{\E\oplus\F}=\nabla^\E \oplus \nabla^\F$. In general, $\C(\Omega \A)$ is not an abelian category. For example, if $\Omega^k \A=0$ for all $k\geq1$, then $\C(\Omega \A)$ is simply the category of fgp modules over $\A$, which is generally not an abelian category. In fact we can easily prove a necessary condition on a graded commutative dga $\Omega \A$ if $\C(\Omega \A)$ is abelian. We will call a differential graded commutative algebra $\Omega \A$ {\em connected} if $\A$ is connected ({\em i.e.} contains no non-trivial projections). 

\begin{lemma}\label{lem-abelian_implies_something_like_Q}
Let $\Omega \A$ be a connected graded commutative dga and suppose that $\C(\Omega \A)$ is abelian. Let $a\in \A$ and suppose that $da=a\omega$ for some $\omega\in\Omega^1 \A$. Then $a$ is either 0 or invertible.
\end{lemma}
\begin{proof}
Consider the two objects $( \A,d+\omega)$ and $( \A,d)$ of $\C(\Omega \A)$. Since $da=a\omega$ we have a commuting diagram
\[\begin{tikzcd}
 \A\arrow[r,"a"] \arrow[d,"d+\omega"] & \A \arrow[d,"d"]\\
\Omega^1 \A\arrow[r,"a"]& \Omega^1 \A
\end{tikzcd}\]
where $a$ denotes the multiplication by $a$. So multiplication by $a$ is a morphism between these objects. We then get a short exact sequence $0\to\ker(a)\to  \A\xrightarrow a\im(a)\to0$. This is a short exact sequence of $ \A$-modules, and since $\im(a)$ is an fgp $ \A$-module, it is split, and we get $ \A\cong \ker(a)\oplus\im(a)$. Since $ \A$ is connected this means that either $\im(a)=0$, which means that $a=0$, or $\im(a)= \A$, which means that $a$ is invertible.
\end{proof}

We will now define a slightly stronger condition on $ \A$, and we will show later that this is a sufficient condition for the category $\C(\Omega \A)$ to be abelian.

\begin{definition}
Let $\Omega \A$ be a $*$-dga with $\A \subseteq A$, densely. We say that $\Omega \A$ satisfies {\em property Q} if it satisfies the following condition:

for all $a\in \A$ with $a\geq 0$ and all $a_1,\ldots,a_s\in \A$ with all $|a_i|\leq a$, and all $\omega_1,\ldots,\omega_s\in \Omega^1 \A$:

if $da=\sum_{i=1}^sa_i\omega_i$, then either $a=0$ or $a$ is invertible.
\end{definition}
If $\Omega \A$ is graded commutative, this easily implies the conclusion in Lemma \ref{lem-abelian_implies_something_like_Q}: if $da=a\omega$, we have $a^*a\geq 0$ and $d(a^*a)=d(a)^*a+a^*da=aa^*(\omega+\omega^*)$, so $aa^*$ is 0 or invertible, hence $a$ is 0 or invertible. It also implies that $ \A$ is connected, in the sense that there are no non-trivial projections: if $p\in \A$ is a projection, then $dp=d(p^2)=2pdp$, so $(1-2p)dp=0$ and multiplying by $1-2p$ gives $dp=0$. Then $p$ should be 0 or invertible, so any projection is 0 or 1.


\subsubsection{Quantum metric differential graded algebras}
We will now show that property Q holds for quantum metric differential graded algebras. First we recall the notion of a compact quantum metric space, introduced by Rieffel \cite{compactquantummetricspaces}. Let $A$ be a $C^*$-algebra and let $L$ be a seminorm on $A$ that takes finite values on a dense subalgebra $ \A$. We think of $L$ as a Lipschitz norm. This defines a metric on the state space $\S(A)$ by Connes' distance formula \cite[Ch. VI.1]{connes}: for $\chi,\psi\in\S(A)$ we have
\[d_L(\chi,\psi)=\sup\{|\chi(a)-\psi(a)|,a\in \A,L(a)\leq 1\}.\]
This metric then defines a topology on the state space. We already had the weak-$*$ topology, so it is natural to make the following definition:
\begin{definition}
Let $A$ be a unital $C^*$-algebra and let $L$ be a seminorm on $A$ taking finite values on a dense subalgebra. Then $A$ is called a \emph{compact quantum metric space} if the topology on $\S(A)$ induced by the metric $d_L$ coincides with the weak-$*$ topology.
\end{definition}

Now we go back to the case that we have a $*$-dga $\Omega \A$ and $ \A$ is a dense subset of a unital $C^*$-algebra $A$. Suppose that a norm $\norm{\cdot}$ is given on $\Omega^1 \A$, satisfying the inequality $\norm{a\omega}\leq \norm a\cdot \norm{\omega}$ for $a\in  \A,\omega\in\Omega^1 \A$. This defines a seminorm $L$ on $ \A$ by $L(a)=\norm{da}$. The space $\Omega \A$ is called a \emph{quantum metric dga} if $A$ is a compact quantum metric space with this seminorm.
\begin{remark}
Note that any compact quantum metric space gives rise to a quantum metric dga. Indeed, it was realised in \cite{compactquantummetricspaces} that the Lipschitz norm can be obtained as $L(a) = \| [D,a]\|$ for a suitable operator $D$ on a Hilbert space. The space of Connes' differential forms \cite[Ch. VI]{connes} is then a quantum metric dga. 
\end{remark}
Also note that if $\Omega \A$ is a quantum metric dga, the same holds for $Z_g(\Omega \A)$, by \cite[Proposition 2.3]{compactquantummetricspaces}.

\begin{lemma}\label{lem-0_or_invertible}
Let $\Omega \A$ be a graded commutative quantum metric dga and suppose that $ \A\cap A^\times= \A^\times$. Then $\Omega \A$ satisfies property Q.
\end{lemma}
\begin{proof}
Let $a\in \A$ with $a\geq 0$, and let $a_1,\ldots,a_s\in  \A$ with all $|a_i|\leq a$, and $\omega_1,\ldots,\omega_s\in \Omega^1 \A$ satisfying $da=\sum_{i=1}^sa_i\omega_i$. By scaling we may assume that $0\leq a\leq 1$. Define the polynomial $p_n(x)=\sum_{k=1}^n\frac1k(1-x)^k$, which is the truncation of the power series of $-\log(x)$. Then we have
\[dp_n(a)=p_n'(a)da=-\sum_{i=1}^sa_i\omega_i\sum_{k=1}^n(1-a)^{k-1}.\]
For each $1\leq i\leq s$ we have
\[\left|a_i\sum_{k=1}^n(1-a)^{k-1}\right|\leq\left|a\sum_{k=1}^n(1-a)^{k-1}\right|=|1-(1-a)^n|\leq 1.\]
So we get
\[\norm{dp_n(a)}\leq \sum_{i=1}^s\norm{\omega_i},\]
in particular the norm of $dp_n(a)$ is bounded as $n\to \infty$.

If $a$ is neither 0 nor invertible in $A$, there are points $\chi,\psi$ in the Gelfand spectrum of $A$ satisfying $\chi(a)=0$ and $\psi(a)=t>0$. Then $\chi(p_n(a))=\sum_{i=1}^n\frac1k\to \infty$ as $n\to \infty$, while $\psi(p_n(a))=\sum_{i=1}^n\frac1k(1-t)^k\to -\log(t)$ as $n\to \infty$. We get
\[d(\chi,\psi)\geq \frac{|\chi(p_n(a))-\psi(p_n(a))|}{\norm{dp_n(a)}} \to \infty\]
so $d(\chi,\psi)=\infty$. But the metric $d$ should give the weak-$*$ topology on the spectrum, and the spectrum is connected, so this is a contradiction. So either $a=0$ or $a\in A^\times$, and in the second case $a\in  \A\cap A^\times= \A^\times$.
\end{proof}

\subsubsection{Proof of abelianness}
In the rest of this section, we will show that if a graded commutative dga $\Omega \A$ satisfies property Q, then $\C(\Omega \A)$ is an abelian category. Suppose we have a morphism $\phi: \E\to \F$ in the category $\C(\Omega \A)$. We have to show that $\ker(\phi),\im(\phi),\coker(\phi)$ are also in the category. The most difficult part is to show that these are finitely generated projective modules.
\begin{lemma}\label{lem-pseudoinverse}
Let $\phi: \E\to  \F$ be a morphism between fgp $\A$-modules. Then the following are equivalent:
\begin{itemize}
\item The $\A$-modules $\ker(\phi),\im(\phi),\coker(\phi)$ are fgp.
\item There is an $\A$-module homomorphism $\phi^+: \F\to  \E$ satisfying $\phi\phi^+\phi=\phi$.
\end{itemize}
\end{lemma}
\begin{proof}
Suppose that $\ker(\phi),\im(\phi),\coker(\phi)$ are fgp. Then the short exact sequence $0\to\ker(\phi)\to  \E\to \im(\phi)\to 0$ is split, so $ \E\cong \ker(\phi)\oplus \im(\phi)$. The short exact sequence $0\to \im(\phi)\to \F\to \coker(\phi)\to 0$ is also split, so $\F\cong \im(\phi)\oplus \coker(\phi)$. The map $\phi$ then corresponds to the map $\ker(\phi)\oplus \im(\phi)\to \im(\phi)\oplus \coker(\phi)$ sending $(a,b)$ to $(b,0)$. We can then choose the map $\phi^+:\im(\phi)\oplus\coker(\phi)\to \ker(\phi)\oplus\im(\phi)$ sending $(c,d)$ to $(0,c)$. It is then easy to check that $\phi\phi^+\phi=\phi$ (and also $\phi^+\phi\phi^+=\phi^+$).

Now suppose there is an $ \A$-linear map $\phi^+:\F\to  \E$ satisfying $\phi\phi^+\phi=\phi$. The surjection $\F\to \coker(\phi)$ admits a splitting, sending the equivalence class $[f]$ to $f-\phi\phi^+(f)$. This is well-defined because $\phi\phi^+\phi=\phi$. So the short exact sequence $0\to\im(\phi)\to\F\to\coker(\phi)\to0$ is split, giving $\F\cong\im(\phi)\oplus\coker(\phi)$. So $\im(\phi)$ and $\coker(\phi)$ are fgp. Then the short exact sequence $0\to\ker(\phi)\to \E\to\im(\phi)\to0$ is also split, giving $ \E\cong\ker(\phi)\oplus\im(\phi)$. So $\ker(\phi)$ is also fgp.
\end{proof}
\begin{remark}
If $\phi:\mathbb C^n\to \mathbb C^n$ and $\phi^+:\mathbb C^n\to \mathbb C^n$ satisfy $\phi\phi^+\phi=\phi$ and $\phi^+\phi\phi^+=\phi^+$, and $\phi\phi^+$ and $\phi^+\phi$ are self-adjoint then $\phi^+$ is uniquely determined, and is called the \emph{Moore-Penrose pseudoinverse} \cite{campbell_1977}.
\end{remark}

In the case that $\ker(\phi),\im(\phi),\coker(\phi)$ are finitely generated projective it is easy to construct connections on these modules.

\begin{lemma}\label{lem-connections_on_ker_im_coker}
Let $\phi: \E\to \F$ be a morphism in $\C(\Omega \A)$ and suppose that $\ker(\phi)$, $\im(\phi)$ and $\coker(\phi)$ are finitely generated projective. Then there are natural induced connections on $\ker(\phi)$ and $\coker(\phi)$. There are then also natural induced connections on $\coim(\phi)=\coker(\ker(\phi))$ and $\im(\phi)=\ker(\coker(\phi))$ and these are compatible with the natural isomorphism of $ \A$-bimodules $\coim(\phi)\xrightarrow\sim\im(\phi)$.
\end{lemma}
\begin{proof}
We have the isomorphisms $ \E\cong\ker(\phi)\oplus\im(\phi)$ and $\F\cong \im(\phi)\oplus\coker(\phi)$ as in the proof of Lemma \ref{lem-pseudoinverse}. This makes the commuting diagram
\[\begin{tikzcd}
\ker(\phi)\arrow[r]\arrow[dr] &  \E \arrow[r,"\phi"] \arrow[d,"\sim"] & \F \arrow[d,"\sim"]\\
&\ker(\phi)\oplus \im(\phi) \arrow[r] &\im(\phi)\oplus\coker(\phi).
\end{tikzcd}\]
The lower horizontal map sends an element $(a,b)$ to $(b,0)$. Tensoring this diagram with $\Omega^1 \A$ gives
\[\begin{tikzcd}[column sep = small]
\ker(\phi)\tensor\Omega^1 \A\arrow[r]\arrow[dr] & \Omega \E\tensor\Omega^1 \A \arrow[r,"\phi\tensor\Omega^1 \A"] \arrow[d,"\sim"] & \F\tensor\Omega^1 \A\arrow[d,"\sim"]\\
&\ker(\phi)\tensor\Omega^1 \A \oplus \im(\phi)\tensor\Omega^1 \A \arrow[r] &\im(\phi)\tensor\Omega^1 \A\oplus \coker(\phi)\tensor\Omega^1 \A.
\end{tikzcd}\]
Now we see that $\ker(\phi)\tensor\Omega^1 \A$ is the kernel of the lower horizontal map, so it is also the kernel of $\phi\tensor\Omega^1 \A$. We also know that $\coker(\phi)\tensor\Omega^1 \A$ is the cokernel of the map $\phi\tensor\Omega^1 \A$ because tensoring with $\Omega^1 \A$ is right exact. Now consider the diagram
\[\begin{tikzcd}
\ker(\phi)\arrow[r]& \E\arrow[d,"\nabla^ \E"]\arrow[r,"\phi"]&\F\arrow[d,"\nabla^\F"]\arrow[r]&\coker(\phi)\\
\ker(\phi)\tensor\Omega^1 \A\arrow[r]& \E\tensor\Omega^1 \A\arrow[r,"\phi\tensor\Omega^1 \A"]&\F\tensor\Omega^1 \A\arrow[r]&\coker(\phi)\tensor\Omega^1 \A.
\end{tikzcd}\]
We see that $\nabla^ \E$ induces a map $\nabla^{\ker(\phi)}:\ker(\phi)\to\ker(\phi)\tensor\Omega^1 \A$ and that $\nabla^\F$ induces a map $\nabla^{\coker(\phi)}:\coker(\phi)\to\coker(\phi)\tensor\Omega^1 \A$. It is easy to check that these satisfy the Leibniz rules. These are the connections we want.

Now consider the diagram
\[\begin{tikzcd}
 \E\arrow[r]\arrow[d,"\nabla^ \E"]&\coim(\phi)\arrow[r,"\sim"]\arrow[d,"\nabla^{\coim(\phi)}"] & \im(\phi)\arrow[r]\arrow[d,"\nabla^{\im(\phi)}"] & \F\arrow[d,"\nabla^\F"]\\
 \E\tensor\Omega^1 \A\arrow[r]& \coim(\phi)\tensor\Omega^1 \A \arrow[r,"\sim"] &\im(\phi)\tensor\Omega^1 \A\arrow[r]&\F\tensor\Omega^1 \A.
\end{tikzcd}\]
The connections $\nabla^{\coim(\phi)}$ and $\nabla^{\im(\phi)}$ are constructed as above as connections on the cokernel of the morphism $\ker(\phi)\to \E$ and the kernel of the morphism $\F\to\coker(\phi)$, respectively. We know that the left and the right square of the diagram commute and simple diagram chasing shows that the middle square commutes as well. So the isomorphism $\coim(\phi)\xrightarrow\sim\im(\phi)$ becomes an isomorphism in $\C(\Omega \A)$.
\end{proof}

\begin{definition}
Let $M\in M_n( \A)$ be a matrix with coefficients in a commutative algebra $ \A$. Let $\chi_M\in  \A[x]$ be the characteristic polynomial, with coefficients $(-1)^mD_m(M)$, so 
\[\chi_M(x)=x^n-D_{n-1}(M)x^{n-1}+D_{n-2}(M)x^{n-2}-\ldots+(-1)^nD_0(M).\]
\end{definition}
Note that $D_0(M)=(-1)^n\det(M)$ and $D_{n-1}=\Tr(M)$. In general, $D_m(M)$ is the sum of the determinants of $m\times m$ square submatrices.

We need the inequality below involving $D_m$. Its proof is an easy calculation after diagonalising $M^*M$, and not very interesting. Its proof can be found in the appendix. The term $2\re(M^*[M,K])$ will appear in the proof of Theorem \ref{thm-abelian}.
\begin{lemma}\label{lem-inequality-dm}
For $M,K\in M_n(\mathbb C)$ we have
\[\left|\frac d{dt}_{|t=0}D_m(M^*M+t\cdot 2\re(M^*[M,K]))\right|\leq 4n\norm{K}_{\HS}D_m(M^*M).\]
Here $2\re(M^*[M,K])=M^*[M,K]+(M^*[M,K])^*$ and $\norm{K}_{\HS}$ denotes the Hilbert--Schmidt norm of $K$.
\end{lemma}
\begin{proof}
See Lemma \ref{lem-proof_inequality} in the appendix.
\end{proof}

\begin{remark}\label{rem-inequality_in_MnA}
Both sides of this inequality are continuous functions of the entries of $M$ and $K$. The inequality then still holds for $M,K\in M_n( \A)$, since it can be checked at any point in the spectrum (the Hilbert--Schmidt norm is then $\norm{K}_{\HS}=\norm{\Tr(K^*K)}^\frac12$).
\end{remark}

We are now ready to prove that $\C(\Omega \A)$ is an abelian category. We need to show that any $\phi: \E\to\F$ has a finitely generated projective kernel, image and cokernel. In the first part of the proof we reduce to the case $\phi: \A^{ n}\to  \A^{ n}$. In the second part we prove that each term in the characteristic polynomial of $\phi^*\phi$ is either zero or invertible. Lastly we use this to prove that $\phi$ has a pseudo-inverse (as in Lemma \ref{lem-pseudoinverse}).

\begin{theorem}\label{thm-abelian}
Let $\Omega \A$ be a graded commutative $*$-dga satisfying property Q. Then the category $\C(\Omega \A)$ is abelian.
\end{theorem}
\begin{proof}
Let $\phi: \E\to \F$ be a morphism in $\C(\Omega \A)$. We will show that $\ker(\phi),\im(\phi),\coker(\phi)$ are finitely generated projective $ \A$-modules, and then we are done by Lemma \ref{lem-connections_on_ker_im_coker}.

There is a projective module $\G$ with $ \E\oplus\F\oplus \G\cong \A^n$. We can write $\G\cong p\A^n$ for a projection $p\in\End_{ \A}(\A^n)$. Then we can define a connection $\nabla^\G:p\A^n\to p(\Omega^1 \A)^{ n}$ by $\nabla^\G(g)=pdg$. It is easy to check that this defines a connection on $\G$ (it is called the Grassmannian connection). This makes $(\G,\nabla^\G)$ an object of $\C(\Omega \A)$ and it also defines a connection on the direct sum module $ \E\oplus\F\oplus\G$.

Now the map
\[\begin{pmatrix}0&0&0\\\phi&0&0\\0&0&0\end{pmatrix}: \E\oplus \F\oplus \G\to  \E\oplus \F\oplus \G\]
is a morphism in $\C(\Omega \A)$. Its kernel is $\ker(\phi)\oplus \F\oplus \G$, its image is $0\oplus \im(\phi)\oplus 0$ and its cokernel is $ \E\oplus \coker(\phi)\oplus \G$. So it is enough to show that these are finitely generated projective.  Therefore it is enough to prove: for a connection $\nabla:\A^n\to (\Omega^1 \A)^{ n}$ and a morphism $\phi:\A^n\to \A^n$ that commutes with $\nabla$, the kernel, image and cokernel of $\phi$ are fgp modules.

The connection $\nabla:\A^n\to (\Omega^1 \A)^{ n}$ can be written as $\nabla=d+\kappa$, where $\kappa:\A^n\to (\Omega^1 \A)^{ n}$ is an $ \A$-linear function. We can view $\kappa$ as an $n\times n$ matrix with coefficients in $\Omega^1 \A$. The induced connection on $\Hom_{ \A}(\A^n,\A^n)$, which we still call $\nabla$, satisfies
\[\nabla(\langle f,e\rangle)=\langle\nabla(f),e\rangle+\langle f,\nabla(e)\rangle\]
for $f\in \Hom_{ \A}(\A^n,\A^n)$ and $e\in \A^n$. So
\[d(\langle f,e\rangle)+\kappa \langle f,e\rangle=\langle\nabla(f),e\rangle+\langle f,de\rangle+\langle f,\kappa e\rangle\]
and this gives
\[\nabla(f)=df+[\kappa,f].\]
Since $\phi:\A^n\to \A^n$ commutes with the connection, we know that $\nabla(\phi)=0$ so we conclude that
\[d\phi=[\phi,\kappa]\]
where $\phi$ is viewed as an element of $M_n( \A)$. We get
\[d(\phi^*\phi)=\phi^*d(\phi)+d(\phi)^*\phi=2\re(\phi^*[\phi,\kappa]).\]
Now let $a_m=D_m(\phi^*\phi) \in  \A$ be the $m$-th term of the characteristic polynomial of $\phi^*\phi$ (up to sign). Write $\kappa=\sum_{i=1}^sK_i\omega_i$ with $K_i\in M_n( \A)$ and $\omega_i\in \Omega^1 \A$. We get
\begin{align*}
da_m&=dD_m(\phi^*\phi)\\
&=\frac d{dt}_{|t=0}D_m(\phi^*\phi+td(\phi^*\phi))\\
&=\frac d{dt}_{|t=0}D_m(\phi^*\phi+t\cdot 2\re(\phi^*[\phi,\kappa]))\\
&=\sum_{i=1}^s\frac d{dt}_{|t=0}D_m(\phi^*\phi+t\cdot 2\re(\phi^*[\phi,K_i]))\omega_i\\
&=\sum_{i=1}^sa_i\omega_i
\end{align*}
where
\[a_i=\frac d{dt}_{|t=0}D_m(\phi^*\phi+t\cdot 2\re(\phi^*[\phi,K_i])).\]
By Lemma \ref{lem-inequality-dm} and Remark \ref{rem-inequality_in_MnA} we get $|a_i|\leq 4\norm{K_i}_{\HS}a_m$. We can now apply property Q (if we put the factor $4\norm{K_i}_{\HS}$ in the $\omega_i$) to conclude that $a_m$ is either 0 or invertible.

Now consider the smallest $m$ for which $a_m\neq 0$ (note  $a_n=1$ so this $m$ exists). Then $a_m$ is invertible. The characteristic polynomial of $\phi^*\phi$ is now $\chi_{\phi^*\phi}(x)=x^n-a_{n-1}x^{n-1}+\ldots+(-1)^{n-m}a_mx^m$. Let $p(x)=x^{-m}\chi_{\phi^*\phi}=x^{n-m}-a_{n-1}x^{n-m-1}+\ldots+(-1)^{n-m}a_m$. By Cayley--Hamilton we know $\chi_{\phi^*\phi}(\phi^*\phi)=0$, so $\phi^*\phi p(\phi^*\phi)$ is nilpotent, and also self-adjoint, so $\phi^*\phi p(\phi^*\phi)=0$. Then $(\phi p(\phi^*\phi))\cdot (\phi p(\phi^*\phi))^*=0$, so in fact already $\phi p(\phi^*\phi)=0$. Let $q(x)=\frac{1+(-1)^{n-m-1}a_m^{-1}p(x)}x\in  \A[x]$ and set $\phi^+=q(\phi^*\phi)\phi^*$. Then we have
\begin{align*}
\phi\phi^+\phi&=\phi q(\phi^*\phi)\phi^*\phi
=\phi(1+(-1)^{n-m-1}a_m^{-1}p(\phi^*\phi))=\phi.
\end{align*}
By Lemma \ref{lem-pseudoinverse} it follows that the kernel, image and cokernel of $\phi$ are finitely generated projective. This concludes the proof of the theorem.
\end{proof}

\begin{corollary}
With the same conditions as in the theorem, the category $\C_{\flat}(\Omega \A)$ is also abelian.
\end{corollary}
\begin{proof}
The category $\C_{\flat}(\Omega \A)$ is a full subcategory of $\C(\Omega \A)$. The kernel, image and cokernel of a morphism in $\C_{\flat}(\Omega \A)$ are again in $\C_{\flat}(\Omega \A)$ because the connections constructed in Lemma \ref{lem-connections_on_ker_im_coker} are flat if $\nabla^ \E$ and $\nabla^\F$ are flat.
\end{proof}

\begin{corollary}
Let $\Omega \A$ be a $*$-dga such that the graded centre $Z_g(\Omega\A)$ satisfies property Q. Then both categories $\C(\Omega \A)$ and $\C_{\flat}(\Omega\A)$ are abelian.
  \end{corollary}
\proof
This follows directly from Theorem \ref{thm-abelian} and the equivalence between $\C(\Omega\A)$ and $\C(Z_g(\Omega\A))$ proved in Theorem \ref{thm-reduction_to_commutative_case}. 
\endproof

\subsection{Definition of the fundamental group}
In this section we will define the fundamental group of a dga satisfying suitable analytical conditions. 
We will first complete the proof that $\C_{\flat}(\Omega \A)$ is a Tannakian category, after which the fundamental group is defined as the group of automorphisms of the fibre functor.
Since, we have already proven that $\C_{\flat}(\Omega \A)$ is a rigid tensor category, and under some conditions on $ \A$, that it is abelian, what is left to show is that $\End(\Omega \A)=\mathbb C$ and constructing a fiber functor $\omega:\C_{\flat}(\Omega \A)\to \Vec$, where $\Vec$ is the category of finite-dimensional vector spaces over $\mathbb C$. 

\begin{lemma}
Let $\Omega \A$ be a graded commutative dga satisfying property Q. Then the algebra of endomorphisms is $\End(\Omega \A)=\mathbb C$.
\end{lemma}
\begin{proof}
Let $\theta:\Omega \A\to \Omega \A$ be an isomorphism. Since $\theta$ is bilinear, for all $\alpha\in\Omega \A$ we have $\theta(\alpha)=\alpha\theta(1)$. So $\theta$ is determined by $a=\theta(1) \in \A$. Since $\theta$ has to commute with the connection we get $da=d(\theta(1))=\theta(d(1))=0$. Let $\lambda$ be a complex number in the spectrum of $a$. Then we have $d(a-\lambda)=0$, but $a-\lambda$ is not invertible. Since $\Omega \A$ satisfies property Q it follows that $a=\lambda\in\mathbb C$.
\end{proof}

For the fibre functor, pick a point $p$ in the Gelfand spectrum $\widehat{A}$ of a commutative $A$ (that contains $ \A$ densely). Then our fibre functor is given by sending a bimodule $\E$ to the localisation of its centre at $p$. This is defined as $ \E\tensor_{ \A}\mathbb C$, where the $ \A$-module structure on $\mathbb C$ is given by $p$. Note that this depends on a choice of a point in the Gelfand spectrum. This point plays a similar role as the base point of the usual fundamental group.

\begin{lemma}
  \label{lma:fibrefunctor}
Let $\Omega \A$ be a graded commutative dga that satisfies property Q and let $p\in \widehat{A}$. There is a faithful exact fibre functor $\omega:\C_{\flat}(\Omega \A)\to \Vec$ sending $ \E$ to $ \E_p$.
\end{lemma}
\begin{proof}
Let $( \E,\nabla^ \E)$ and $(\F,\nabla^\F)$ be objects of $\C_{\flat}(\Omega \A)$ and let $\phi: \E\to\F$ be a morphism commuting with the connections. Since $\phi$ is $ \A$-linear, this induces a map $ \E_p\to \F_p$, showing that $\omega$ is functorial.

To show that $\omega$ is faithful, suppose that $\phi_p=0$. Since $\C_{\flat}(\Omega \A)$ is abelian we know that $\im(\phi)$ is an fgp module. Now look at $\im(\phi)\tensor_{ \A} A$. This is an fgp module over the C$^*$-algebra $A$, which corresponds to a vector bundle on $\widehat A$. It is zero at $p$, and the rank is locally constant, and $\widehat A$ is connected, so $\im(\phi)\tensor_ \A A=0$. Since $\im(\phi)$ is projective it is flat, and $\im(\phi)\hookrightarrow\im(\phi)\tensor_{ \A} A$ is an injection, so also $\im(\phi)=0$. We conclude that $\phi=0$.

The fibre functor is exact because a localisation is always exact.
\end{proof}

\begin{theorem}
  Let $\Omega \A$ be a $*$-dga such that $Z_g(\Omega \A)$ satisfies property Q. Then the category $\C_{\flat}(\Omega \A)$ can be equipped with the structure of a neutral Tannakian category.
  \end{theorem}
\begin{proof}
  We already know that the category $\C_\flat(\Omega\A)$ is an abelian rigid tensor category. The reduction to the graded commutative case (Theorem \ref{thm-reduction_to_commutative_case}) then allows us to apply the previous two Lemma's to complete the proof. 
  \end{proof}
The fibre functor $\omega: \C_{\flat}(\Omega \A) \to \Vec$ is then of course given as the composition of the functor defined in Theorem \ref{thm-reduction_to_commutative_case} and the fibre functor in Lemma \ref{lma:fibrefunctor}.

We thus derive from eg. \cite[Theorem 2.11]{Del90} that the category $\C_{\flat}(\Omega \A)$ is equivalent to the category of representations of an algebraic group scheme, which allows us to make the following definition.  
\begin{definition}\label{def-fundamental_group}
Let $\Omega \A$ be a dga such that $Z_g(\Omega \A)$ satisfies property Q. Let $p\in \widehat {Z_g( \A)}$. Then we define $\pi^1(\Omega \A,p)$ to be the group scheme of automorphisms of the fibre functor $\omega:\C_{\flat}(\Omega \A)\to \Vec$ at $p$.
\end{definition}



Note that in practice we will simply recognise the category $\C_{\flat}(\Omega \A)$ as being equivalent to the category of representations of some (topological) group whose pro-algebraic completion is the fundamental group.

\begin{example}
We have seen in Example \ref{ex:manifold} that for a connected manifold $M$ without boundary we have $\C_{\flat}(\Omega M) \cong \Rep (\pi_1(M))$. Hence $\pi^1( \Omega M)$ is the pro-algebraic completion of $\pi_1(M)$. 
\end{example}

\begin{example}\label{ex-[0,1]}
Consider the graded commutative $*$-dga $\Omgb [0,1]$, where $\Omega^0[0,1]=C^\infty[0,1]$ and $\Omega^1[0,1]=C^\infty[0,1]dx$, $\Omega^i[0,1]=0$ for $i\geq 2$. Each fgp module over $C^\infty[0,1]$ is a free module. We will show explicitly that each flat connection on a free module is isomorphic to the trivial connection $d$, thereby showing that $\C_{\flat}(\Omgb[0,1])$ is equivalent to the category of vector spaces and $\pi^1(\Omgb[0,1])$ is trivial as we would expect.

Let $W\tensor C^\infty[0,1]$ be a module over $C^\infty[0,1]$ where $W$ is a vector space, and let $\nabla:W\tensor C^\infty[0,1]$ be a connection. Write $\nabla=d+\omega$, where $\omega:W\tensor C^\infty[0,1]\to W\tensor\Omega^1[0,1]$ is a linear map. Let $\alpha:W\tensor C^\infty[0,1]\to W\tensor C^\infty[0,1]$ be an isomorphism. Then the diagram
\[\begin{tikzcd}
W\tensor C^\infty[0,1] \arrow[r,"\nabla"] \arrow[d,"\alpha" ] & W\tensor \Omega^1[0,1] \arrow[d,"\alpha"]\\
W\tensor C^\infty[0,1] \arrow[r,"d"]& W\tensor \Omega^1[0,1]
\end{tikzcd}\]commutes if and only if $\nabla=\alpha^{-1}\circ d\circ \alpha$. If we see $\alpha$ as an element of $\End(W)\tensor C^\infty[0,1]$ we can take the derivative of it, which satisfies $d(\alpha)=d\circ\alpha-\alpha\circ d \in \End(W)\tensor C^\infty[0,1]$. So the diagram above commutes when $\nabla=\alpha^{-1}\circ (d(\alpha)+\alpha\circ d)=d+\alpha^{-1}d(\alpha)$.

Hence we are looking for an invertible element $\alpha\in\End(W)\tensor C^\infty[0,1]$ satisfying $d(\alpha)=\alpha\omega$. 
The solution of this equation is given by a path-ordered exponential, namely, we set $\alpha = \sum_{n=0}^\infty\alpha_n$ where
$\alpha_0=1$ and where recursively $\alpha_{n+1}(t)=\int_0^t\alpha_n\omega$ for $n\geq0$. 
It follows easily by induction that $\norm{\alpha_n(t)}\leq \frac{t^n\norm{\omega}^n}{n!}$, so the series converges. Since $d\alpha_{n+1}=\alpha_n\omega$ we get $d\alpha=\alpha\omega$. In a similar way we can construct $\alpha'$ satisfying $d\alpha'=-\omega\alpha'$, and it is easy to see that this is the inverse of $\alpha$.

So $(W\tensor C^\infty[0,1],\nabla)$ is isomorphic to $(W\tensor C^\infty[0,1],d)$ and we conclude that $\pi^1(\Omgb[0,1])=0$.
\end{example}

\begin{example}\label{ex-M2C}
Let $\Omega^k \A=M_2(\mathbb C)$ for all $k$. Let $d$ be given by taken the graded commutator with the matrix $D=\left(\begin{smallmatrix}1&0\\0&-1\end{smallmatrix}\right)$. Explicitly, $d$ is given in even degrees by $\left(\begin{smallmatrix}a&b\\c&d\end{smallmatrix}\right)\to\left(\begin{smallmatrix}0&-2b\\2c&0\end{smallmatrix}\right)$ and in odd degrees by $\left(\begin{smallmatrix}a&b\\c&d\end{smallmatrix}\right)\to\left(\begin{smallmatrix}2a&0\\0&-2d\end{smallmatrix}\right)$. This is the noncommutative space corresponding to the set of two points that are identified.

The graded centre of this dga is just $Z_g(\Omega \A)=\mathbb C \oplus 0 \oplus \mathbb C \oplus 0 \oplus \cdots$ where $\mathbb C$ is embedded diagonally in $M_2(\mathbb C)$. Then $\C_{\flat}(Z_g(\Omega \A))$ is just equivalent to the category of vector spaces. The fundamental group is trivial.
\end{example}

\begin{example}\label{ex-C_C}
Consider the following graded commutative dga: let $\Omega^0\B=\mathbb C,\Omega^1\B=\mathbb C$ and $\Omega^n\B=0$ for $n\geq 2$, with $d=0$. Note that there is a unique point in $\widehat B$. An fgp $\Omega \B$-bimodule is simply a finite-dimensional vector space $V$ and a connection is a $\mathbb C$-linear map $\nabla:V\to V$, and it is always flat. So the category $\C_{\flat}(\Omega \B)$ is equivalent to the category of vector spaces with an endomorphism. This is in turn equivalent to the category of continuous representations of $\mathbb R$: the vector space $V$ with the endomorphism $\alpha$ corresponds to the representation $\pi:\mathbb R\to \End(V)$, $\pi(t)=\exp(t\alpha)$. All continuous representations of $\mathbb R$ are of this form by Lemma \ref{lem-reps_of_R} in the appendix. Since the category of all continuous representations of a topological group $\mathbb R$ is equivalent to the category of representations of the algebraic hull of $\mathbb R$ ({\em cf.} \cite[Ex. 2.33]{DM82}), the fundamental group of $\Omega \B$ is then the algebraic hull of $\mathbb R$.
\end{example}

\begin{example}\label{ex-tensor_product_with_C_C}
Let $\Omega \B$ be the dga from the previous example and let $\Omega \A$ be any graded commutative dga satisfying property Q, and let $p\in\widehat A$. Consider the graded tensor product $\Omega \A'=\Omega \A\tensor\Omega \B$. This is a dga, its modules are given by $(\Omega^0 \A')=\Omega \A$ and $\Omega^n \A'=\Omega^n \A \oplus \Omega^{n-1} \A$ for $n\geq 1$. Let $\E$ be an fgp $\A$-bimodule. A connection over $(\Omega \A')$ is given by a map $\nabla:\E \to \E\tensor\Omega^1 \A\oplus \E$. Write $\nabla=\nabla_ \A\oplus\nabla_\B$, where $\nabla_ \A:\E\to  \E\otimes\Omega^1 \A$ is a connection over $\Omega \A$ and $\nabla_\B:\Omega \E\to\Omega \E$ is an $\A$-linear map. The curvature of the connection $\nabla^2:\E\to  \E\tensor \Omega^2 \A \oplus \E \tensor\Omega^1 \A$ is then given by $\nabla_\A^2\oplus [\nabla_\B,\nabla_ \A]$. So $\nabla$ is flat if and only if $\nabla_ \A$ is flat and commutes with the endomorphism $\nabla_\B$. We get a series of equivalences
\begin{align*}
\C_{\flat}(\Omega \A')\simeq&\{\text{objects of }\C_{\flat}(\Omega \A)\text{ with an endomorphism}\}\\
\simeq&\{\text{objects of }\Rep(\pi^1(\Omega \A))\text{ with an endomorphism}\}\\
\simeq&\Rep(\pi^1(\Omega \A)\times\mathbb R).
\end{align*}
Hence the fundamental group of $\Omega \A'=\Omega \A\tensor\Omega \B$ at $p$ is (the algebraic hull of) $\pi^1(\Omega \A,p)\times\mathbb R$.
\end{example}

This leads to the following useful general result. 
\begin{proposition}
  \label{prop-product_with_wedge_V}
  Let $V$ be an $n$-dimensional vector space and consider the graded algebra $\exterior V$ with differential $d=0$. Then $\pi^1(\exterior V)$ is the algebraic hull of $\mathbb R^n$. More generally, if $\Omega \A$ is any graded commutative dga which satisfies property Q and $p\in \widehat {A}$, then the fundamental group of $\Omega \A\tensor\exterior V$ at $p$ is the algebraic hull of $\pi^1(\Omega \A,p)\times\mathbb R^n$.
\end{proposition}
\proof
The dga $\exterior V$ is isomorphic to the $n$-fold tensor product of the dga in Example \ref{ex-C_C}, so by Example \ref{ex-tensor_product_with_C_C} we see that $\pi^1(\exterior V)$ is the algebraic hull of $\mathbb R^n$. Similar reasoning leads to the second statement. 
\endproof

\section{Some properties of the fundamental group}
\label{sect:props}
In this section we will establish some of the crucial properties that one would like a fundamental group to possess. This includes base point invariance, functoriality, homotopy invariance and Morita invariance.

\subsection{Base point invariance}
As one might expect, different base points give rise to isomorphic fundamental groups, at least provided that they are joined by a smooth path. In this case we will also simply write $\pi^1(\Omgb\mathcal A)$, omitting the base point from the notation.

\begin{proposition}\label{prop-basepointinvariance}
Let $\Omgb\mathcal A$ be a $*$-dga such that $Z_g(\Omgb\mathcal A)$ satisfies property Q. Let $p,q\in\widehat{Z_g(\mathcal A)}$ be base points. Suppose we have a $*$-homomorphism $\gamma:Z_g(\Omgb\mathcal A)\to\Omgb[0,1]$ satisfying $\ev_0\circ\gamma=p,\ev_1\circ\gamma=q$ where $\ev_t$ denotes evaluation at $t$. Then there exists an isomorphism $\pi^1(\Omgb\mathcal A,p)\cong \pi^1(\Omgb\mathcal A,q)$.
\end{proposition}
\begin{proof}
We may assume without loss of generality that $\Omgb\mathcal A$ is graded commutative. Consider the following diagram:
\[\begin{tikzcd}[column sep=huge,row sep=huge]
\mathcal C_{\flat}(\Omega \mathcal A)\arrow[r,"\gamma"]\arrow[d,shift right,"p" left]\arrow[d,shift left,"q"] & \mathcal C_{\flat}(\Omega[0,1]) \arrow[dl,shift right,"v_0" above]\arrow[dl,shift left,"v_1" below]\\
\Vec
\end{tikzcd}\]
Here the horizontal map $\gamma$ denotes, by abuse of notation, the functor that sends $\mathcal E$ to $\mathcal E\tensor_\gamma C^\infty[0,1]$, and similar for the maps $p$ and $q$. The functor $v_0$ sends $\mathcal E$ to $\mathcal E_0$ and similar for $v_1$. Let $n:\mathcal C_{\flat}(\Omgb[0,1])\to \Vec$ be the functor sending $(\mathcal E,\nabla)$ to the vector space $\ker(\nabla)$. By Example \ref{ex-[0,1]}, we have $\ker(\nabla)\tensor C^\infty[0,1]\xrightarrow\sim \mathcal E$. Therefore the map $\ker(\nabla)\to \mathcal E_0$ given by localisation at $0$ is an isomorphism. This gives a natural isomorphism between $n$ and $v_0$. Similarly we have a natural isomorphism from $n$ to $v_1$. So the functors $v_0$ and $v_1$ are naturally isomorphic, and it follows that the fibre functors from $p$ and $q$ are also naturally isomorphic. Then their group schemes of automorphisms are isomorphic as well, giving the isomorphism $\pi^1(\Omgb\mathcal A,p)\cong \pi^1(\Omgb\mathcal A,q)$.
\end{proof}

It is not known whether base point invariance holds without the additional assumption in Proposition \ref{prop-basepointinvariance}

\subsection{Functoriality of the fundamental group}
For graded commutative spaces there is a good notion of functoriality for the fundamental group. In this section we will address the question for which maps between dga's there is an induced map between the fundamental groups. 

We start by observing that the fundamental group $\pi^1(\Omega \A,p)$ is defined in terms of the dga $\Omega \A$ and a character $p$ on the center $Z_g(\Omega \A)$, interpreted as the base point. Now, if $\phi:  \Omega \A\to\Omega \B$ is a map of dga's one can only expect functoriality on the corresponding fundamental groups to have any meaning at all if base points are mapped to base points. In other words, characters of the center should be mapped to characters of the center. In other words, it is crucial to demand that the map $\phi$ maps the center to the center. Under such conditions, we can establish the following functorial property of the fundamental group of dga's. 
\begin{proposition}\label{lem-functoriality_graded_commutative}
Let $\Omega \A$ and $\Omega \B$ be $*$-dga's such that their graded centers satisfy property Q. Let $\phi:\Omega \A\to\Omega \B$ be a degree 0 algebra morphism satisfying $\phi(d\alpha)=d(\phi(\alpha))$ for all $\alpha\in\Omega \A$ and that $\phi(Z_g(\Omega \A)) \subseteq Z_g(\Omega \B)$. Let $q\in \widehat {Z_g( \B)}$ and $p=\phi^*(q)\in \widehat {Z_g(\A)}$. Then $\phi$ induces a map $\pi^1\phi:\pi^1(\Omega \B,q)\to \pi^1(\Omega \A,p)$.
\end{proposition}
\begin{proof}
  Since $\C(\Omega \A)$ is equivalent to $\C(Z_g(\Omega \A))$ (and the same for $\Omega \B$ instead of $\Omega \A$) and $\phi$ induces a map of dga's from $Z_g(\Omega \A) \to Z_g(\Omega \B)$ we may assume without loss of generality that our dga's are graded commutative. Hence, if $ \E$ is an fgp $ \A$-module then $ \E\tensor_\A  \B$ is an fgp $ \B$-module. A flat connection $\nabla: \E\to  \E\tensor\Omega^1 \A$ gives a flat connection $\tilde\nabla: \E\tensor_{ \A} \B\to \E\tensor_ \A  \Omega^1 \B$, given by $\tilde\nabla(e\tensor b)=\nabla(e)b+e\tensor db$. So we get a map $\C(\Omega \A)\to \C(\Omega \B)$. It is easy to see that it is functorial and also that it commutes with the fibre functors. Then every automorphism of the fibre functor $\C(\Omega \B)\to\Vec$ can be pulled back to an automorphism of the fibre functor $\C(\Omega \A)\to\Vec$. So we get a map $\pi^1\phi:\pi^1(\Omega \B,q)\to\pi^1(\Omega \A,p)$.
\end{proof}

\subsection{Homotopy invariance}

In this subsection we will show that homotopic maps $\phi_0,\phi_1:(\Omgb\mathcal B,p)\to (\Omgb\mathcal A,q)$ give rise to the same map $\pi^1\phi_0=\pi^1\phi_1:\pi^1(\Omgb \A,q)\to\pi^1(\Omgb \B,p)$.


\begin{lemma}\label{lem-nabla-isomorphic-d}
Let $\mathcal E\tensor C^\infty[0,1]$ be an fgp module over $\mathcal A\tensor C^\infty[0,1]$, where $\mathcal E$ is an fpg module over $\mathcal A$. Let $\nabla:\mathcal E\tensor C^\infty[0,1]\to\mathcal E\tensor\Omega^1[0,1]$ be a connection over $\Omgb[0,1]$ that is $\mathcal A$-linear. Then there is an $\mathcal A\tensor C^\infty[0,1]$-linear isomorphism $\alpha:\mathcal E\tensor C^\infty[0,1]\to \mathcal E\tensor C^\infty[0,1]$ such that $\alpha\circ \nabla =d\circ\alpha$.
\end{lemma}
\begin{proof}
The proof is essentially the same as that of Example \ref{ex-[0,1]}. Write $\nabla=d+\omega$ with $\omega:\mathcal E\tensor C^\infty[0,1]\to\mathcal E\tensor\Omega^1[0,1]$ an $\mathcal A\tensor C^\infty[0,1]$-linear map. We want to construct $\alpha$ such that $d\alpha=\alpha\omega$, viewing $\alpha$ as an element of $\End_\mathcal A \E\tensor C^\infty[0,1]$. Let $\alpha_0=1$, and define recursively $\alpha_{n+1}(t)=\int_0^t\alpha_n\omega$. Then $\alpha=\sum_{n=0}^\infty\alpha_n$ is well-defined and invertible as in Example \ref{ex-[0,1]} and it satisfies $d\alpha=\alpha\omega$.
\end{proof}

\begin{lemma}\label{lem-modules_Atimes01}
Let $\mathcal E$ be an fgp module over $\mathcal A\tensor C^\infty[0,1]$ with a connection $\nabla_2:\mathcal E\to  \mathcal E\tensor\Omega^1[0,1]$ over $\Omgb[0,1]$ that is $\mathcal A$-linear. Then $\ker(\nabla_2)$ is an fgp $\mathcal A$-submodule of $\mathcal E$ and we have a natural isomorphism $\ker(\nabla_2)\tensor C^\infty[0,1]\xrightarrow\sim \mathcal E$ given by multiplication.
\end{lemma}
\begin{proof}
Let $\mathcal F$ be another fgp module over $\mathcal A\tensor C^\infty[0,1]$ such that $\mathcal E\oplus \mathcal F$ is free. We can choose a (Grassmannian) connection $\nabla^\mathcal F:\mathcal F\to \mathcal F\tensor\Omega^1[0,1]$ over $\Omgb[0,1]$ that is $\mathcal A$-linear. Now the free module $\mathcal E\oplus\mathcal F$ has the connection $\nabla_2\oplus\nabla^\mathcal F$. By lemma \ref{lem-nabla-isomorphic-d} we may identify $\mathcal E\oplus\mathcal F$ with $W\tensor\mathcal A\tensor C^\infty[0,1]$ for some vector space $W$ with the connection given by $d:W\tensor\mathcal A\tensor C^\infty[0,1]\to W\tensor\mathcal A\tensor\Omega^1[0,1]$. Now note that $\ker(d)=W\tensor\mathcal A$ and the map $\ker(d)\tensor C^\infty[0,1]\to \mathcal E\oplus\mathcal F$ is an isomorphism. It follows that $\ker(\nabla_2)=\ker(d)\cap\mathcal E$ is an fgp $\mathcal A$-module and the map $\ker(\nabla_2)\tensor C^\infty[0,1]\to \mathcal E$ is an isomorphism.
\end{proof}

We can use this lemma to prove a version of homotopy invariance for $\pi^1$. First we give the definition of homotopy for morphisms of dga's.

\begin{definition}\label{def-homotopy}
Let $\Omgb\mathcal A$ and $\Omgb\mathcal B$ be $*$-dga's whose graded centres satisfy property Q and let $p\in\widehat {Z_g(\A)},q\in\widehat {Z_g(\B)}$ be base points. Two maps $\phi_0,\phi_1:(\Omgb\mathcal B,q)\to (\Omgb\mathcal A,p)$ are called homotopic if there exists a map
\[H:\Omgb\mathcal B\to \Omgb\mathcal A\tensor\Omgb[0,1]\]
that satisfies the following conditions:
\begin{itemize}
    \item the map $H$ sends graded centre to graded centre;
    \item it satisfies $\ev_t\circ H=\phi_t$ for $t =0,1$, where $\ev_t:\Omgb[0,1]\to\mathbb C$ denotes evaluation at $t$;
    \item the diagram
\[\begin{tikzcd}
Z_g(\Omgb\mathcal B) \arrow[r,"H"]\arrow[d,"q"] & Z_g(\Omgb\mathcal A)\tensor\Omgb[0,1]\arrow[d,"{p\tensor\Omgb[0,1]}"]\\
\mathbb C\arrow[r,"z\to\text{const}_z"] & \Omgb[0,1]
\end{tikzcd}\]
commutes.
\end{itemize}
\end{definition}

\begin{remark}
Considering the diagram above in grade 0, we see that it means that $H$ pulls back the point $p\times t\in \widehat{Z_g(\A)}\times [0,1]$ to the point $q$ for all $t\in[0,1]$. If $Z_g(\Omega^1 \B)$ is generated by elements of the form $b_0\ db_1$, this is an equivalent condition.
\end{remark}

\begin{theorem}\label{thm-homotopy_invariance}
Let $\Omgb\mathcal A$ and $\Omgb\mathcal B$ be $*$-dga's whose graded centres satisfy property Q and let $p\in\widehat {Z_g(\A)},q\in\widehat {Z_g(\B)}$ be base points. Let $\phi_0,\phi_1:(\Omgb\mathcal B,q)\to(\Omgb\mathcal A,p)$ be homotopic maps. Then $\pi^1\phi_0=\pi^1\phi_1:\pi^1(\Omgb\mathcal A,p)\to \pi^1(\Omgb\mathcal B,q)$.
\end{theorem}
\begin{proof}
Let $H:\Omgb\mathcal B\to \Omgb\mathcal A\tensor\Omgb[0,1]$ be as in Definition \ref{def-homotopy}. We may assume without loss of generality that $\Omgb\mathcal A$ and $\Omgb\mathcal B$ are graded commutative. Let $\mathcal E$ be an fgp module over $\mathcal A\tensor C^\infty[0,1]$ with a flat connection $\nabla$. We can write $\nabla=\nabla_1\oplus\nabla_2:\mathcal E\to \mathcal E\tensor\Omega^1\mathcal A\oplus \mathcal E\tensor\Omega^1[0,1]$ where $\nabla_1:\mathcal E\to \mathcal E\tensor\Omega^1 \mathcal A$ is a connection over $\Omgb\mathcal A$ that is $C^\infty[0,1]$ linear while $\nabla_2:\mathcal E\to\mathcal E\tensor\Omega^1[0,1]$ is a connection over $\Omgb[0,1]$ that is $\mathcal A$-linear. By Lemma \ref{lem-modules_Atimes01} we know that $\ker(\nabla_2)$ is an fgp $\Omgb\mathcal A$-module, and it is easy to check that $\nabla_{1\mid\ker(\nabla_2)}$ is a connection on this module. So we can consider the functor
\[n:\mathcal C_{\flat}(\Omgb\mathcal A\tensor\Omgb[0,1])\to \mathcal C_{\flat}(\Omgb\mathcal A)\]
given by $n(\mathcal E,\nabla)=(\ker(\nabla_2),\nabla_{1\mid\ker(\nabla_2)})$. Consider also the functor
\[v_0:\mathcal C_{\flat}(\Omgb\mathcal A\tensor\Omgb[0,1])\to \mathcal C_{\flat}(\Omgb\mathcal A)\]
given by $v_0(\mathcal E,\nabla)=(\mathcal E_{\ev_0},\nabla_ {\ev_0})$ where $\ev_0:\Omgb\mathcal A\tensor\Omgb[0,1]\to\Omgb\mathcal A$ denotes evaluation at 0. By Lemma \ref{lem-modules_Atimes01} we have a natural isomorphism $\eta:n\to v_0$, sending $\ker(\nabla_2)$ to $\mathcal E_{\ev_0}$ by the composition $\ker(\nabla_2)\to \mathcal E\to \mathcal E_{\ev_0}$.

Taking $\Omgb\mathcal A=\mathbb C$ in the above we get similar maps $n',v_0':\mathcal C_{\flat}(\Omgb[0,1])\to \Vec$, and a natural isomorphism $\eta':n'\to v_0'$. Now consider the diagram
\[\begin{tikzcd}[column sep=huge,row sep=huge]
\mathcal C_{\flat}(\Omgb\mathcal A\tensor\Omgb[0,1]) \arrow[dd,"{p\tensor\Omgb[0,1]}"] \arrow[r,bend left,"n",""{name=U,below}] \arrow[r,bend right,"v_0" below,""{name=V}] & \mathcal C_{\flat}(\Omgb\mathcal A) \arrow[dd,"p"]\\
\\
\mathcal C_{\flat}(\Omgb[0,1]) \arrow[r,bend left,"n'",""{name=U',below}] \arrow[r,bend right,"v_0'" below,""{name=V'}] & \Vec
\arrow[Rightarrow,from=U,to=V,"\eta"]
\arrow[Rightarrow,from=U',to=V',"\eta'"]
\end{tikzcd}\]
Here $p\tensor\Omgb[0,1]$ denotes by abuse of notation the functor sending $\mathcal E$ to $\mathcal E\tensor_{p\tensor\Omgb[0,1]}\Omgb[0,1]$, and similar for $p$. It is easy to see that $p\circ n = n'\circ (p\tensor\Omgb[0,1])$ and $p\circ v_0=v_0'\circ(p\tensor\Omgb[0,1])$. Moreover, we have $(p\tensor\Omgb[0,1])^*\eta'=p_*\eta$.
Now consider the diagram
\[\begin{tikzcd}
\mathcal C_{\flat}(\Omgb\mathcal B) \arrow[r,"H"] \arrow[d,"q"] & \mathcal C_{\flat}(\Omgb\mathcal A\tensor\Omgb[0,1]) \arrow[d,"{p\tensor\Omgb[0,1]}"]\\
\Vec \arrow[r,"F"]& \mathcal C_{\flat}(\Omgb[0,1]).
\end{tikzcd}\]
Here we used similar abuse of notation as in the diagram above, and $F$ is the functor sending a vector space $W$ to $(W\tensor C^\infty[0,1],d)$. This diagram commutes because $H$ is a homotopy. So we have $p_*H^*\eta=H^*(p\tensor\Omgb[0,1])^*\eta'=q^*F^*\eta'$. Now $F^*\eta'$ is a natural isomorphism between $n'\circ F=\id:\Vec\to\Vec$ and $v_0'\circ F=\id:\Vec\to\Vec$, and it is in fact the identity. So $H^*\eta$ is a natural isomorphism from $n\circ H$ to $\phi_0$, and $p_*H^*\eta:q\to q$ is the identity. Of course we get a similar natural isomorphism for evaluation at 1 instead of evaluation at 0, and composing these we get a natural isomorphism $\mu:\phi_0\to\phi_1$ satisfying $p_*\mu=\id:q\to q$.  Then the maps $\phi_0,\phi_1$ induce the same map $\pi^1\phi_0=\pi^1\phi_1=\pi^1(\Omgb\mathcal A,p)\to \pi^1(\Omgb\mathcal B,q)$.
\end{proof}

It follows directly that $\pi^1$ is an invariant for homotopy equivalence:
\begin{corollary}
Let $\Omgb\mathcal A$ and $\Omgb\mathcal B$ be $*$-dga's whose graded centres satisfy property Q and let $p\in\widehat {Z_g(\A)},q\in\widehat {Z_g(\B)}$ be base points. Let $\phi:(\Omgb\mathcal A,p)\to(\Omgb\mathcal B,q)$ and $\psi:(\Omgb\mathcal B,q)\to(\Omgb\mathcal A,p)$ be morphisms such that $\phi\circ\psi$ and $\psi\circ\phi$ are homotopic to the identity on $\Omgb\mathcal B$ and $\Omgb\mathcal A$ respectively. Then $\pi^1(\Omgb\mathcal A,p)$ is isomorphic to $\pi^1(\Omgb\mathcal B,q)$.
\end{corollary}
\begin{proof}
It follows from the theorem that $\pi^1(\phi)$ and $\pi^1(\psi)$ are inverse to each other.
\end{proof}

\subsection{Invariance under Morita equivalence}
We now address the question whether $\pi^1$ is invariant under Morita equivalence of the underlying dga's. Since we work with differential graded algebras as well as with $C^*$-algebras, for both of which there exist notions of Morita equivalence, let us make more precise what we mean.

Let $(R,d)$ be a differential graded {\em ring}. We denote by $\Mod^\dg_{R}$ the category of all differential graded right $R$-modules $(M,d)$ and with morphisms {\em all} graded module morphisms, not only those of degree 0. We write $d$ instead of $\nabla$ here to distinguish these differentials from the flat bimodule connections considered before. They satisfy the right Leibniz rule ({\em  cf.} Equation \eqref{eq:right-leibniz}):
$$
d(mr) = d(m) r + (-1)^{|m|} m dr ; \qquad (m \in M, r \in R).
$$
It then follows that the morphisms $\Hom((M,d), (N,d))$ become differential graded modules (over the differential graded ring $\End((M,d))$), with
\begin{equation}
  \label{eq:induced-diff}
df (m) = d (f(m) ) - (-1)^{|f|} f(dm); \qquad (m \in M).
\end{equation}
Thus, the category $\Mod^\dg_{R}$ is a so-called {\em differential graded category}, or {\em dg-category}.

\begin{definition}
Two differential graded rings $R$ and $S$ are called {\em dg-Morita equivalent} if the categories $\Mod^\dg_{R}$ and $\Mod^\dg_{S}$ are equivalent. 
\end{definition}

It is a classical result in Morita theory that Morita equivalent rings have isomorphic centers (see for instance  \cite[Remark 18.43]{Lam}). We prove an analogue for dg-Morita equivalent differential graded rings. Recall that the graded center of an additive category is given by all graded natural transformations $\eta$ from the identity functor to itself ({\em cf.} \cite[Section 4]{NS17} and references therein), {\em i.e.} for all dg-modules $M,N$ and $f  \in \Hom((M,d), (N,d))$ there is a (graded) commuting diagram
\[\begin{tikzcd}
M \arrow[r,"f"] \arrow[d,"\eta_M"] & N \arrow[d,"\eta_N"]\\
M \arrow[r,"f"] & N
\end{tikzcd}\]
in the sense that $f \circ \eta_N = (-1)^{|\eta||f|} \eta_M \circ f$.

\begin{proposition}
The center of the category $\Mod^\dg_{R}$ is a (graded commutative) differential graded ring which is isomorphic to the graded center $Z_g(R)$ of $R$. 
  \end{proposition}
\proof
Since $\eta_M \in \Hom((M,d), (M,d))$ we can define $(d\eta)_M = d( \eta_M)$ using Equation \eqref{eq:induced-diff}. This turns the graded center $C(\Mod^\dg_{R})$ of $\Mod^\dg_{R}$ into a dg ring. Let us then show that $C(\Mod^\dg_{R})$ is isomorphic to $Z_g(R)$. 

As in \cite[Remark 18.43]{Lam} we define a map
\begin{align*}
\rho:  Z_g(R )& \to C(\Mod^\dg_{R})\\
  r &\mapsto \eta^{(r)}
\end{align*}
where $\eta^{(r)}$ is given by right multiplication by $r$, that is to say, $\eta^{(r)}_M(m) = (-1)^{|m||r|}m r$ for all $m \in M$. The map $\rho$ is a map of dg rings since $|\eta^{(r)}| = |r|$ and 
\begin{align*}
  ( d \eta^{(r)})_M (m) &= d( \eta^{(r)}_M (m)) - (-1)^{|r|} \eta^{(r)}_M (d m) \\
  &= (-1)^{|m||r| }d(m r) -(-1)^{(|m|+2)|r|}  (dm) r \\
  &= (-1)^{|m|(|r|+1)} m (d r) = \eta^{(d r)}_M (m).
  \end{align*}
It is also clearly injective. To prove that $\rho$ is surjective take any $\eta \in C(\Mod^\dg_{R})$ and consider first $\eta_{R} : R \to R$. This map satisfies $\eta_{R}(s) = \eta_{R}(1) s =  r s$ where we have set $r := \eta_{R}(1) \in R$. On the other hand, multiplication on the left on $R$ by an element in $R$ is a morphism in $\Mod^\dg_{R}$ so that by graded naturality we also have $\eta_{R}(s) = (-1)^{|s||\eta|} s \eta_{R}(1) = (-1)^{|r||s|}s r$. Hence $r \in Z_g(R)$ and $\eta_{R} =\eta_{R}^{(r)}$. For an arbitrary module $M$ in $\Mod^\dg_{R}$ consider the following (graded) commuting diagram:
\[\begin{tikzcd}
R \arrow[r,"f"] \arrow[d,"\eta_{R}"] & M \arrow[d,"\eta_M"]\\
R \arrow[r,"f"] & M
\end{tikzcd}\]
where for $m \in M$ we have defined $f(r) = m r$, a morphism of graded right $R$-modules of degree $|f| = |m|$. Then
$$
\eta_M(m) = \eta_{M}(f(1)) = (-1)^{|\eta| |f|}f(\eta_{R}(1)) = (-1)^{|r| |m|}f(r) =  (-1)^{|m||r|} m r
$$
so that we may conclude that $\eta_M = \eta_M^{(r)}$ for all dg-modules $M$ and hence that $\rho$ is surjective.
\endproof

\begin{theorem}
Suppose that $\Omega \A$ and $\Omega \B$ are two dg-Morita equivalent dga's whose graded centers satisfy property Q. 
Then
  \begin{enumerate}
\item There is an isomorphism of dga's $\phi: Z_g(\Omega\A) \to Z_g(\Omega\B)$,
\item The induced map $\pi^1\phi: \pi^1 (\Omega \B,q) \to \pi^1 (\Omega \A,p)$ is an isomorphism where $p= \phi^* (q) \in\widehat{Z_g(\Omega\A)}$ for $q \in \widehat{Z_g(\Omega\B)}$. 
  \end{enumerate}

  \end{theorem}
\proof
From the proof of the previous proposition we find that $Z_g(\Omega\A) \cong C(\Mod_{\Omega\A}^{\dg})$ and also $Z_g(\Mod_{\Omega\B}^\dg) \cong C(\Mod_{\Omega\B}^{\dg})$. Since the categories $\Mod_{\Omega\B}^\dg$ and $\Mod_{\Omega\B}^{\dg}$ are equivalent, their graded centers are isomorphic, which proves a). Functoriality of $\pi^1$ as proved in Proposition \ref{lem-functoriality_graded_commutative} in combination with the fact that $\pi^1 (\Omega \A,p) \cong \pi^1(Z_g(\Omega \A),p)$ and $\pi^1 (\Omega\B,q) \cong \pi^1(Z_g(\Omega \B),q)$ ({\em cf.} Theorem \ref{thm-reduction_to_commutative_case})  then proves b). 
\endproof

\begin{remark}
  It is an interesting question to see whether $\pi^1$ is invariant under {\em derived} Morita equivalence as well. For instance, in \cite[Prop. 9.2]{Ric89} or \cite[Prop. 6.3.2]{KZ98} it is shown that the center of a ring is invariant under derived Morita equivalence. For the generalization to differential graded rings we refer to the notes \cite{Kel98,Sch03}, see also \cite[Remark 5.6]{NS17}. 
  \end{remark}

\section{Examples: toric noncommutative manifolds}
\label{sect:ex}
\subsection{Noncommutative tori}
In this section we will consider the noncommutative torus, also called the rotation algebra. Let $\A_\theta$ be the rotation algebra, as studied by Rieffel \cite{Rieffel-rotation} and Connes \cite{Connes-torus}, and described in \cite[Ch.12]{GraciaBondia}. For any real number $\theta$ we define it to be the following $*$-algebra
$$
\A_\theta := \big \{ \sum_{m,n}  a_{mn} u^m v^n : (a_{mn}) \in \S(\mathbb Z^2) \big\}
$$
where $u,v$ are unitaries that satisfy $uv=\lambda vu$ where $\lambda=e^{2\pi i\theta}$. Note that the algebra $\A_\theta$ has a natural $\mathbb Z^2$-grading where $u^mv^n$ has degree $(m,n)$. In fact, this degree is related to the action $\alpha$ of a 2-dimensional torus $\mathbb T^2$ by automorphisms on $ \A_\theta$ given by $\alpha_t(u^mv^n) = e^{i m t_1 + n t_2}u^mv^n$. 

Let us now introduce the dga for the noncommutative torus, given by noncommutative differential forms $\Omega \A_\theta$. The elements of $\Omega^1( \A_\theta)$ are of the form $adu+bdv$ with $a,b\in \A_\theta$ and they satisfy $udu=du\cdot u, udv=\lambda dv\cdot u,vdu=\overline\lambda du\cdot v,vdv=dv\cdot v$ and $du\cdot dv=-\lambda dv\cdot du$. The elements of $\Omega^2 \A_\theta$ are of the form $adudv$ with $a\in \A_\theta$ (see \cite[Sect. 12.2]{GraciaBondia}). The action $\alpha$ extends to the dga as a graded automorphism by demanding that it commutes with the differential.

Note that an integer value of $\theta$ gives back the algebra $C^\infty(\mathbb T^2)$ of the usual torus. The noncommutative torus looks rather differently for $\theta$ irrational and $\theta$ rational. In both cases we will compute the graded centre of $\Omega \A_\theta$ and from there the fundamental group.

\begin{proposition}
  \label{ex-irrational_torus}
  Let $\theta$ be irrational. Then the center of $\Omega \A_\theta$ is trivial and the fundamental group of $\Omega \A_\theta$ is isomorphic to (the algebraic hull of) $\mathbb R^2$.
\end{proposition}
\proof
It is well-known that $Z(\A_\theta)$ is trivial, see for instance \cite[Corl. 12.12]{GraciaBondia}). In fact, this result extends to the differential forms for which we have $Z_g(\Omega \A_\theta)= (\Omega\A_\theta)^{\mathbb T^2}$, the subalgebra of invariant vectors for the action $\alpha$ of $\mathbb T^2$. We see that $Z_g(\Omega^1 A_\theta)=u^{-1}du\cdot \mathbb C \oplus v^{-1}dv \cdot \mathbb C$ and $Z_g(\Omega^2 \A_\theta)=u^{-1}du\cdot v^{-1}dv \cdot \mathbb C$. We conclude that the dga $Z_g(\Omega \A)$ is isomorphic to the dga $\exterior V$ where $V$ is the two-dimensional vector space $u^{-1}du\cdot\mathbb C\oplus v^{-1}dv\cdot\mathbb C$. From Proposition \ref{prop-product_with_wedge_V} we then conclude that the fundamental group of $\Omega \A_\theta$ is isomorphic to (the algebraic hull of) $\mathbb R^2$.
\endproof

Apparently the flat connections on fgp $\Omega\A_\theta$-bimodules correspond to continuous representations of $\mathbb R^2$. We can give the correspondence explicitly. Any continuous representation of $\mathbb R^2$ on a vector space $W$ is given by $(t_1,t_2)\to \exp(t_1\alpha+t_2\beta)$ with $\alpha,\beta \in \End(W)$ commuting endomorphisms of the vector space. The corresponding module is $\Omega \E=W\tensor \Omega \A_\theta$, and the connection is given by
\begin{align*}
&\nabla:W\tensor \Omega \A_\theta\to W\tensor \Omega \A_\theta\\
&\nabla(w\tensor a)=w\tensor da+\alpha(w)\tensor u^{-1}du\cdot a+\beta(w)\tensor v^{-1}dv\cdot a.
\end{align*}
Note that all fgp $\A_\theta$-bimodules are free by Lemma \ref{lem-graded_centre}, as $Z_g( \A_\theta)=\mathbb C$.

\begin{remark}
This should be compared to \cite{SuijlekomMahanta}, which considered the irrational rotation algebra with a holomorphic structure as a noncommutative elliptic curve, with resulting fundamental group equal to (the algebraic hull of) $\mathbb Z$. 
\end{remark}

\begin{proposition}
  \label{ex-rational_torus}
  Let $\theta$ be rational. Then the fundamental group of $\Omega \A_\theta$ coincides with that of the classical manifold $\mathbb T^2$, {\em i.e.} it is (the algebraic hull of) $\mathbb Z^2$.
  \end{proposition}
\proof
We write $\theta=\frac pq$ with $p,q$ coprime integers. Then the centre of $ \A_\theta$ is given by power series in the commuting unitaries $u^q$ and $v^q$ with coefficients of radid decay, and thus isomorphic to the algebra $C^\infty(\mathbb T^2)$ (see also \cite[Corl. 12.3]{GraciaBondia}). Furthermore, we have $Z_g(\Omega^1 \A_\theta)=Z_g(\A_\theta)\cdot u^{-1}du \oplus Z_g(\A_\theta)\cdot v^{-1}dv$. This is also generated by $u^q$ and $v^q$ and their derivations, as $u^{-q}du^q=q\cdot u^{-1}du$ and $v^{-q}dv^q=q\cdot v^{-1}dv$. So $Z_g(\Omega^1 \A_\theta)\cong \Omega^1\mathbb T^2$. Similarly, $Z_g(\Omega^2 \A_\theta)\cong \Omega^2\mathbb T^2$. We see that $Z_g(\Omega \A_\theta)\cong \Omega\mathbb T^2$ and hence that the fundamental group of $\Omega \A_\theta$ is the same as that of the classical manifold $\mathbb T^2$.
\endproof

\begin{remark}
For any $\theta$ we have an inclusion map $Z_g(\Omega \A_\theta)\hookrightarrow\Omega \A_0$. This gives a map $\pi^1(\Omega \A_0)\to \pi^1(Z_g(\Omega \A_\theta))=\pi^1(\Omega \A_\theta)$. In the case that $\theta$ is irrational, this comes from the inclusion $\mathbb Z^2\to\mathbb R^2$. In the case that $\theta=\frac pq$ it comes from the multiplication $\mathbb Z^2\xrightarrow{\cdot q}\mathbb Z^2$. With this map we can distinguish the rational rotation algebras for different values of $q$.

For any $\theta$ the flat connections over $\Omega \A_\theta$ correspond to the continuous representations of the group $(\mathbb Z+\theta\mathbb Z)^2$: for irrational $\theta$, this is a dense subgroup of $\mathbb R^2$ which has the same continuous representations as $\mathbb R^2$ (see Lemma \ref{lem-reps_dense_subgroup_R} in the appendix), and for rational $\theta$ we have $(\mathbb Z+\theta\mathbb Z)^2\cong \mathbb Z^2$. The map $\pi^1(\Omega \A_0)\to \pi^1(\Omega \A_\theta)$ is then always given by the inclusion $\mathbb Z^2 \hookrightarrow (\mathbb Z+\theta\mathbb Z)^2$.
\end{remark}

\subsection{Higher-dimensional noncommutative tori}\label{subsection-higher_tori}
The irrational tori can be generalised to higher dimensions, as in \cite{Rieffel-higherdimensionaltori}.
\label{ex-higher_dimensional_tori} Let $\Theta$ be a skew-symmetric $n\times n$ matrix with coefficients in $\mathbb R$. We use the notation of \cite[Sect. 12.2]{GraciaBondia}. The algebra $ \A_\Theta$ is generated by unitaries $u_1,\ldots,u_n$ satisfying $u_ku_l=e^{2\pi i\Theta_{kl}}u_lu_k$. We can also write this as $u_ku_l=\tau(e_k,e_l)^2u_lu_k$ where $\tau:(\mathbb Z^n)^2\to \mathbb C$ is the two-cocycle defined by $\tau(r,s)=\exp(\pi i r^t\Theta s)$ and $e_k\in \mathbb Z^n$ denotes the $k$-th unit vector. A general term in $\A_\Theta$ is a power series expansion in the $u_k$ with coefficients in the Schwartz space $\S(\mathbb Z^n)$. It is the noncommutative interpretation of the quotient of $\mathbb R^n$ by $\mathbb Z^n+\Theta\mathbb Z^n$, which is simply $\mathbb R^n/\mathbb Z^n$ if $\Theta=0$. Note that we get the two-dimensional noncommutative torus back by taking $n=2$ and the matrix $\Theta=\left(\begin{smallmatrix}0&\theta\\-\theta&0\end{smallmatrix}\right)$. 

For any $r\in \mathbb Z$ we define the \emph{Weyl element}
\[u^r=\exp\left(\pi i\sum_{j<k}r_j\Theta_{jk}r_k\right)u_1^{r_1}\cdots u_n^{r_n}.\]
These are linearly independent and generate the algebra $ \A_\Theta$, and they satisfy
\[u^ru^s=\tau(r,s)u^{r+s}.\]

The one-form module $\Omega^1 \A_\Theta$ is free with generators $\{u_k^{-1}du_k\}$. We use these generators because they are in the centre of $\Omega^1  \A_\Theta$. The two-form module $\Omega^ 2 \A_\Theta$ is free with generators $\{u_k^{-1}du_k\cdot u_l^{-1}du_l,k<l\}$, {\em et cetera}.


\begin{proposition}
\label{prop:pi1-n-tori}
The fundamental group $\pi^1(\Omega \A_\Theta)$ of the $n$-dimensional noncommutative torus is the algebraic hull of $\mathbb Z^n + \Theta \mathbb Z^n$. 
\end{proposition}
\proof
Define the lattice
\begin{equation}
  \label{eq:lattice}
\Lambda=\{r\in\mathbb Z^n\mid \Theta r\in \mathbb Z^n\}
\end{equation}
(it is reciprocal to the lattice used in \cite[Sect. 12.2]{GraciaBondia}).  Let $r_1,\ldots,r_m$ be a basis of $\Lambda$ and let $r_{m+1},\ldots,r_n$ be elements of $\mathbb Z^n$ such that $r_1,r_2,\ldots,r_n$ are linearly independent. Then $Z_g(\A_\Theta)$ is generated by the $m$ independent unitaries $u^{r_k}$, so it is isomorphic to $C^\infty(\mathbb T^m)$, the algebra of smooth functions on the (commutative) $m$-torus. 

Let $V\subseteq \Omega^1 \A_\Theta$ be the $m$-dimensional vector space spanned by $u^{-r_k}du^{r_k}, 1\leq k\leq m$. Then $\Omega\mathbb T^m$ is isomorphic to the sub-dga $Z_g( \A_\Theta)\tensor \exterior V$ of $\Omega A_\Theta$. Also, let $W\subseteq \Omega^1\Omega \A_\Theta$ be the $(n-m)$-dimensional vector space spanned by $u^{-r_k}du^{r_k},m+1\leq k\leq n$. Thus
\[
\Omega \A_\Theta= \Omega \A_\Theta\tensor\exterior(V\oplus W).
\]
Since all basis elements of $V$ and $W$ are in the graded centre we get
\[
Z_g(\Omega \A_\Theta)=Z_g(\Omega \A_\Theta)\tensor\exterior(V\oplus W)\cong \Omega\mathbb T^m\tensor\exterior W.
\]
With this we can compute the fundamental group of $\Omega \A_\Theta$. The fundamental group of the torus $\mathbb T^m$ is $\mathbb Z^m$. Then we see from Proposition \ref{prop-product_with_wedge_V} that the fundamental group of $\Omega \A_\Theta$ is the algebraic hull of $\mathbb Z^m\times\mathbb R^{n-m}$. The subgroup $\mathbb Z^n+\Theta\mathbb Z^n\subseteq \mathbb R^n$ is dense in $\{u\in\mathbb R^n\mid r^tu\in\mathbb Z\text{ for all }r\in\Lambda\}$, and this is isomorphic to $\mathbb Z^m\times\mathbb R^{n-m}$. So for any $\Theta$ the fundamental group of $\Omega \A_\Theta$ equals the algebraic hull of $\mathbb Z^n+\Theta\mathbb Z^n$. 
\endproof

\begin{remark}
If $G$ is a (discrete) group that acts freely and properly on $\mathbb R^n$, then $\mathbb R^n\to \mathbb R^n/G$ is the universal cover of $\mathbb R^n/G$ and it is a $G$-principal bundle. In the example above we show that the fundamental group corresponding to the noncommutative realisation of the quotient of $\mathbb R^n$ by $\mathbb Z^n+\Theta\mathbb Z^n$ can be identified with $\mathbb Z^n+\Theta\mathbb Z^n$, which is exactly as we would expect. Of course, since we get the fundamental group from its representations we cannot actually distinguish between $\mathbb Z^n+\Theta\mathbb Z^n$ and $\mathbb Z^m\times\mathbb R^{n-m}$ as fundamental group of $A_\Theta$.
\end{remark}

\subsection{Toric noncommutative manifolds}
It is possible to deform any manifold $M$ that carries an action of a torus in a similar manner to the deformation of the tori described above. This is described in full detail in \cite{CL01,CD02}, and the differential graded algebra that we use is described in \cite[Sect. 12]{CD02}. 

Let $M$ be a manifold and let $\sigma:\mathbb T^n\to \Aut(M)$ be a smooth effective action of the $n$-dimensional torus. This defines an action of $\mathbb T^n$ on $\Omega M$, still denoted by $\sigma$. The deformation of $C^\infty(M)$ is conveniently described as the following fixed-point subalgebra:
$$
C^\infty(M_\Theta):= ( C^\infty(M) \widehat \otimes \A_\Theta)^{\sigma \otimes \alpha^{-1}}= \left\{ T \in C^\infty(M) \widehat \otimes \A_\Theta: (\sigma_t \otimes \alpha_{t}^{-1})(T) = T , \text{ for all } t \in \mathbb T^n \right\}
$$
where $\widehat \otimes$ denotes the (projective) tensor product of Fr\'echet algebras and $\alpha$ is the action of $\mathbb T^n$ on $\A_\Theta$ given by $\alpha_t(u^r)=e^{irt}u^r$. Similarly, a dga is defined by
$$
\Omega M_\Theta:= ( \Omega (M) \widehat \otimes \A_\theta)^{\sigma \otimes \alpha^{-1}}
$$
The differential in $\Omega M_\Theta$ is given by $ d\otimes 1$. 

In fact, we may write any element $\omega \in \Omega M$ as a series expansion in the Weyl elements $u^r \in \Omega \A_\Theta$:
$$
\omega = \sum_{r \in \mathbb Z^n} \omega_r  \otimes u^r
$$
in terms of homogeneous elements $\omega_r \in \Omega M$ for the torus action, {\em i.e.} $\sigma_t ( \omega_r) = e^{i r t}  \omega_r$. This series expansion is convergent with respect to the Fr\'echet topology on $\Omega M$ (see \cite[Ch. 2]{deformationquantisation}). 
Similar to \cite{cacic} we introduce the following subgroup of $\mathbb T^n$ dual to the lattice $\Lambda \subset \mathbb Z^n$ defined by equation \eqref{eq:lattice}:
$$
\Gamma=\{t\in\mathbb T^n\mid  t \cdot  r =0  \mod \mathbb Z \text{ for all }r\in\Lambda\}
$$

\begin{proposition}
The graded center of $\Omega M_\Theta$ is given by
  $$
Z_g(\Omega M_\Theta)  \cong (\Omega M )^{\Gamma}
$$
Consequently, the fundamental group is $\pi^1(\Omega M_\Theta) = \pi^1( (\Omega M )^{\Gamma})$.
  \end{proposition}
\proof
The graded center of $\Omega M_\Theta$ is given by elements of the form $\omega = \sum_{r \in \Lambda} \omega_r \otimes u^r$ as for $r \in \Lambda$ we have $\tau(r,s)=1$ for any $s\in \mathbb Z^n$. We can use the subgroup $\Gamma < \mathbb T^n$ to select these vectors by setting $\sigma_t(\omega) = \omega$ for all $t \in \Gamma$. 
\endproof

\begin{example}
  Let $M=\mathbb T^n$ and let the torus $\mathbb T^n$ act on itself by addition. The algebra $C(M_\Theta)$ is then the same as the noncommutative torus $\Omega \A_\Theta$, and according to the above result we have $\pi^1(\Omega M_\Theta) = \pi^1(\Omega (\mathbb T^n)^\Gamma)$. Let us confront this with Proposition \ref{prop:pi1-n-tori}. We may write
  $$
  \Omega(\mathbb T^n) = C^\infty(\mathbb T^n) \otimes \exterior (V \oplus W) 
  $$
  where $V,W$ are defined as in the proof of Proposition \ref{prop:pi1-n-tori} (but with $\Theta=0$). But then
    $$
  \Omega(\mathbb T^n)^\Gamma =  (C^\infty(\mathbb T^m) \otimes \exterior V) \otimes \exterior W \cong \Omega (\mathbb T^m) \otimes \exterior W.
  $$
Accordingly, we have $\pi^1(\Omega(\mathbb T^n)^\Gamma)= \mathbb Z^{m} \times \mathbb R^{n-m}$, as desired. 
\end{example}

\begin{example}
  Consider the three-sphere $\mathbb S^3$. It is parametrised by two complex numbers $\alpha,\beta$ with $|\alpha|^2+|\beta^2|=1$. The torus acts on this by $(t,s)\cdot (\alpha,\beta)=(\exp(2\pi i t)\alpha,\exp(2\pi i s)\beta)$. 
  Suppose that $\theta=\frac pq$ where $p$ and $q$ have no common factors. Then $\Lambda=q\mathbb Z^2$ and $\Gamma=(\frac1q\mathbb Z^2)/\mathbb Z^2\subset \mathbb T^2$. We get that $\pi^1(\Omega\mathbb S^3_\theta) = \pi^1((\Omega\mathbb S^3)^{(\mathbb Z/q \mathbb Z)^2})$. 
\end{example}
In the case that $\theta$ is irrational we can actually completely calculate the fundamental group of the toric noncommutative 3-sphere. 
\begin{proposition}
  Let $\theta$ be irrational. Then the fundamental group $\pi^1(\Omega\mathbb S^3_\theta)$ is the algebraic hull of $\mathbb R^2$. 
\end{proposition}
\proof
We will compute the fundamental group of the deformed 3-sphere $\mathbb S^3_\theta$ using homotopy invariance. Let $0<\epsilon<\frac12$ and let $\gamma:[0,1]\to[0,1]$ be a smooth function that is 0 on $[0,\epsilon]$ and 1 on $[1-\epsilon,1]$. This gives a map $\mathbb T^2\times\gamma:\mathbb T^2\times[0,1]\to\mathbb T^2\times[0,1]$, which is $\mathbb T^2$-equivariant, using the natural action of $\mathbb T^2$ on $\mathbb T^2\times[0,1]$. This induces a map $(\mathbb T^2\times\gamma)^*:\Omega(\mathbb T^2\times[0,1])\hookrightarrow\Omega(\mathbb T^2\times[0,1])$, which restricts to a map $\Omega(\mathbb T^2\times[0,1])^{\mathbb T^2}\hookrightarrow\Omega(\mathbb T^2\times[0,1])^{\mathbb T^2}$. The functions in the image are constant in a neighbourhood of the edges, and the forms in the image are zero on this neighbourhood. Now consider the map $\phi:\mathbb T^2\times[0,1]\to \mathbb S^3$, sending $((u,v),t)$ to $(u\sin(\frac\pi2t),v\cos(\frac\pi2t))$ for $u,v\in\mathbb S^1$ and $t\in[0,1]$. This is a smooth map so it induces $\phi^*:\Omega\mathbb S^3\hookrightarrow\Omega(\mathbb T^2\times[0,1])$. Since $\phi$ is $\mathbb T^2$-equivariant, this restricts to $\phi^*:\Omega(\mathbb S^3)^{\mathbb T^2}\hookrightarrow\Omega(\mathbb T^2\times[0,1])^{\mathbb T^2}$. The image certainly contains all functions that are constant on a neighbourhood of the edges, and forms that are zero on a neighbourhood of the edges. In particular there is a factorisation $\Omega(\mathbb T^2\times[0,1])\xhookrightarrow{\alpha}\Omega(\mathbb S^3)^{\mathbb T^2}\xhookrightarrow{\phi^*}\Omega(\mathbb T^2\times[0,1])$, where $\alpha$ is the map $(\mathbb T^2\times\gamma)^*$ with restricted codomain. Now let $\tilde\gamma:\mathbb S^3\to\mathbb S^3$ be the function given by $\tilde\gamma(u\sin(\frac\pi2t),v\cos(\frac\pi2t))=(u\sin(\frac\pi2\gamma(t)),v\cos(\frac\pi2\gamma(t)))$. This is a smooth map: this is clear in the domain $t\in(0,1)$ and at $t=0,1$ it follows because $\gamma$ is constant at the edges. This then induces $\tilde\gamma^*:\Omega\mathbb S^3\hookrightarrow\Omega\mathbb S^3$, and since $\tilde\gamma$ is $\mathbb T^2$-invariant this also restricts to $\Omega(\mathbb S^3)^{\mathbb T^2}\hookrightarrow \Omega(\mathbb S^3)^{\mathbb T^2}$. Note that $\tilde\gamma\circ\phi=\phi\circ(\mathbb T^2\times\gamma)$. We get the following commutative diagram:
  \[\begin{tikzcd}[row sep=huge,column sep=huge]
  \Omega(\mathbb T^2\times[0,1])^{\mathbb T^2} & \Omega(\mathbb T^2\times[0,1])^{\mathbb T^2} \arrow[hook,l,"(\mathbb T^2\times\gamma)^*"] \arrow[hook,ld,"\alpha"]\\
  \Omega(\mathbb S^3)^{\mathbb T^2} \arrow[hook,u,"\phi^*"] & \Omega(\mathbb S^3)^{\mathbb T^2} \arrow[hook,l,"\tilde\gamma^*"] \arrow[hook,u,"\phi^*"].
  \end{tikzcd}\]
  Both the maps $(\mathbb T^2\times\gamma)^*$ and $\tilde\gamma$ are homotopic to the identity. This implies by Theorem \ref{thm-homotopy_invariance} that they induce the identity map on the fundamental groups. So $\pi^1(\Omega\mathbb S^3_\theta)=\pi^1(\Omega(\mathbb S^3)^{\mathbb T^2})=\pi^1(\Omega(\mathbb T^2\times[0,1])^{\mathbb T^2})$. Now clearly $\mathbb T^2\times[0,1]$ is homotopy equivalent to $\mathbb T^2$ using maps that are $\mathbb T^2$-equivariant, and it follows that $\pi^1(\Omega(\mathbb T^2\times[0,1]))^{\mathbb T^2})=\pi^1(\Omega(\mathbb T^2)^{\mathbb T^2})$. As in the previous subsection, this fundamental group is the algebraic hull of $\mathbb R^2$. So for irrational $\theta$ the fundamental group of $\mathbb S^3_\theta$ is the algebraic hull of $\mathbb R^2$.
  \endproof

 \section{Conclusion and outlook}

 \label{sect:concl}
 We have defined a notion of connections on finitely generated projective bimodules over a differential graded algebra $\Omega \A$. We have defined the category $\C_{\flat}(\Omega \A)$ of these bimodules with flat connections, and shown that it is equal to the category of flat connections over the graded centre $Z_g(\Omega \A)$. We have constructed a tensor product in this category and shown that it admits dual objects. The category is also abelian for a large class of noncommutative spaces. This was used to define an affine algebraic group scheme, which we called the fundamental group $\pi^1(\Omega \A)$ of the dga. After having established some crucial properties for $\pi^1$,  we computed the fundamental group for noncommutative tori, where we realised that it depends on the deformation parameter. 

The structure we introduced suggests the following some natural, still open problems:
 \begin{enumerate}
   \item Does the fundamental group respect products: is it true that $\pi^1(\Omega \A\tensor\Omega \B,p\tensor q)=\pi^1(\Omega \A,p)\times\pi^1(\Omega\B,q)$?
 \item Is the fundamental group $\pi^1$ a derived Morita invariant? 
   \end{enumerate}




\appendix

\section{Representations of dense subgroups of $\mathbb R$}
We have exploited Tannaka duality to reconstruct a group from the category of its representations, however, this only gives access to the pro-algebraic completion of the group. In other words, via this procedure many groups give rise to the same algebraic group.  We will illustrate this phenomenon for dense subgroups of $\mathbb R$. 

We consider representations of $\mathbb R$ and of dense subgroups. Let $V$ be a finite-dimensional vector space and $\alpha\in\End(V)$, and view $\mathbb R$ as additive topological group. There is a continuous representation $\pi:\mathbb R\to\GL(V)$ given by $\pi(t)=\exp(t\alpha)$. In fact every continuous representation has this form.

\begin{lemma}\label{lem-reps_of_R}
Let $V$ be a finite-dimensional vector space and $\pi:\mathbb R\to \GL(V)$ a continuous representation. Then there is a unique $\alpha\in\End(V)$ such that $\pi(t)=\exp(t\alpha)$ for all $t\in\mathbb R$.
\end{lemma}
\begin{proof}
Since $\pi$ is continuous there is a $\delta>0$ such that for all $t\in\mathbb R$ with $|t|<\delta$ we have $\norm{\pi(t)-1}<1$. Let $0<t<\delta$. The formal power series $\log(x+1)$ has radius of convergence 1, so $\alpha=\frac1t\log(\pi(t))$ is well-defined. Then $\pi(t)=\exp(t\alpha)$. Also $\alpha'=\frac2t\log(\pi(t/2))$ is well-defined. Since $\log(x^2)=2\log(x)$ whenever $|x-1|<1$ and $|x^2-1|<1$, we get $\alpha'=\alpha$. So $\pi(t/2)=\exp(t/2\alpha)$. By induction we get $\pi(t2^{-n})=\exp(t2^{-n}\alpha)$ for positive integers $n$. For integers $m$ it also follows that $\pi(tm2^{-n})=\pi(t2^{-n})^m=\exp(tm2^{-n}\alpha)$. Since the set $\{tm2^{-n}\mid m,n\in\mathbb Z\}$ is already dense in $\mathbb R$ it follows that $\pi(x)=\exp(x\alpha)$ for all $x\in\mathbb R$.
\end{proof}

This result also holds for continuous representations of dense subgroups of $\mathbb R$. This is used in the computation of the fundamental group for noncommutative tori in section 4.
\begin{lemma}\label{lem-reps_dense_subgroup_R}
Let $G\subseteq \mathbb R$ be a dense subgroup of $\mathbb R$. Let $V$ be a finite-dimensional vector space. Each continuous representation $\pi:G\to\GL(V)$ extends to a representation of $\mathbb R$, and is therefore of the form $\pi(t)=\exp(t\alpha)$ for some $\alpha\in\End(V)$.
\end{lemma}
\begin{proof}
Since the representation is continuous there is a $\delta>0$ such that for $t\in[-\delta,\delta]\cap G$ we have $\norm{\pi(t)-1}\leq 1$, so $\norm{\pi(t)}\leq 2$. For all positive integers $n$ we have $\pi(nt)=\pi(t)^n$, and it follows that $\pi$ is bounded on bounded intervals.

Now we show from this that $\pi$ is uniformly continuous on bounded intervals. Let $\epsilon>0$. There is $\delta>0$ such that for all $t\in G\cap(-\delta,\delta)$ we have $\norm{\pi(t)-1}<\epsilon$. Let $M>0$ and let $x,y\in G\cap [-M,M]$ with $|x-y|<\delta$. Then $\norm{\pi(x)-\pi(y)}\leq \norm{\pi(x)}\cdot \norm{1-\pi(y-x)}\leq \norm{\pi(x)}\cdot \epsilon$. Since $\pi$ is bounded on $G\cap [-M,M]$ we see that $\pi$ is uniformly continuous on bounded intervals. Then it can be extended uniquely to a function $\mathbb R\to \End(V)$, and it follows easily that this is still a representation.
\end{proof}

\section{Proof of Lemma \ref{lem-inequality-dm}}
We have put here the proof of Lemma \ref{lem-inequality-dm} that is a bit technical and would otherwise have disrupted the flow of the argument. 

\begin{lemma}\label{lem-proof_inequality}
For $M,K\in M_n(\mathbb C)$ we have
\[\left|\frac d{dt}_{|t=0}D_m(M^*M+t\cdot 2\re(M^*[M,K]))\right|\leq 4n\norm{K}_{\HS}D_m(M^*M).\]
Here $2\re(M^*[M,K])=M^*[M,K]+(M^*[M,K])^*$ and $\norm{K}_{\HS}$ denotes the Hilbert-Schmidt norm of $K$.
\end{lemma}
\begin{proof}
The inequality is invariant under a unitary change of basis of $M$, and $M^*M$ is self-adjoint so we may choose a basis in which $M^*M$ is diagonal, with eigenvalues $\lambda_1,\lambda_2,\ldots,\lambda_n\in \mathbb C$. Now \[D_m(M^*M)=\sum_{|S|=m}\prod_{i\in S}\lambda_i,\]
where $S$ runs over the $m$-element subsets of $\{1,2,\ldots,n\}$. Let
\[M(t)=M^*M+t\cdot M^*[M,K].\]
The $t$-coefficient in the polynomial $P(t)=D_m(M(t))\in \mathbb C[t]$ is $\frac d{dt}_{|t=0}D_m(M(t))$. The matrix $M(t)$ only has multiples of $t$ outside the diagonal. The determinant of an $m\times m$ submatrix is then modulo $t^2$ equal to the product of the values on its diagonal. So
\[P_m(t)=\sum_{S=|m|}\prod_{i\in S}M(t)_{ii} \mod t^2.\]
We have
\[M(t)_{ii}=\lambda_i+t\left(\lambda_iK_{ii}-\sum_{j,l=1}^n(M^*)_{ij}K_{jl}M_{li}\right).\]
The $t$-coefficient in $P_m(t)$ is then
\[\frac d{dt}_{|t=0}D_m(M(t))=\sum_{|S|=m}\sum_{i\in S}\left(\lambda_iK_{ii}-\sum_{j,l=1}^n(M^*)_{ij}K_{jl}M_{li}\right)\prod_{s\in S\setminus\{i\}}\lambda_s.\]
For all $i,j,l$ we have $(M^*)_{ij}M_{li}\leq \frac12(|M_{ji}|^2+|M_{li}|^2)\leq \sum_{k=1}^n|M_{ki}|^2=\lambda_i$. So we get
\begin{align*}
    \left|\frac d{dt}_{|t=0}D_m(M(t))\right|&=\left|\sum_{|S|=m}\sum_{i\in S}\left(\lambda_iK_{ii}-\sum_{j,l=1}^n(M^*)_{ij}K_{jl}M_{li}\right)\prod_{s\in S\setminus\{i\}}\lambda_s\right|\\
    &\leq\sum_{|S|=m}\sum_{i\in S}\left|\lambda_iK_{ii}-\sum_{j,l=1}^n(M^*)_{ij}K_{jl}M_{li}\right|\prod_{s\in S\setminus\{i\}}\lambda_s\\
    &\leq \sum_{|S|=m}\sum_{i\in S}\left(\lambda_i|K_{ii}|+\sum_{j,l=1}^n\lambda_i|K_{jl}|\right)\prod_{s\in S\setminus\{i\}}\lambda_s\\
    &\leq \left(\sum_{i=1}^n|K_{ii}|+\sum_{j,l=1}^n|K_{jl}|\right)\sum_{|S|=m}\prod_{s\in S}\lambda_s\\
    &\leq 2n\norm{K}_{\HS}D_m(M^*M).
\end{align*}
Now we conclude
\begin{align*}
\left|\frac d{dt}_{|t=0}D_m(M^*M+t\cdot 2\re(M^*[M,K]))\right|&=\\
\left|2\re\left(\frac d{dt}_{|t=0}D_m(M(t))\right)\right|&\leq 4n\norm K_{\HS}D_m(M^*M).\qedhere
\end{align*}
\end{proof}


\end{document}